\documentclass[11pt]{amsart}
\usepackage{geometry,amsthm,amsmath,amssymb, graphicx, float, enumerate}
\usepackage{amsmath}
\usepackage[makeroom]{cancel}
\usepackage{scalefnt}
\usepackage{mathrsfs}
\usepackage{color}

\usepackage{charter}
\restylefloat{table}
\restylefloat{figure}

\swapnumbers
\newtheorem{thm}{Theorem}[section]

\newtheorem{cor}[thm]{Corollary}

\newtheorem{lemma}[thm]{Lemma}
\newtheorem{prop}[thm]{Proposition}

\theoremstyle{definition}
\newtheorem{defn}[thm]{Definition}
\theoremstyle{remark}

\newtheorem{ex}[thm]{Example}

\newcommand\jh[1]{{#1}}

\def\M{{\mathcal M}_{N,K}}

\def\Rx{R^{\alpha}_{x}}

\def\fam_Uabg{\Omega}
\def\Uch{W\Phi_{ch}^{\a,\b}}

\def\Uch{\Phi_{ch}^{\a,\b}}
\def\Usum{\Phi_{sum}^\eta}
\def\wUsum{\widehat{\Phi}_{sum}^\eta}
\def\Udiag{\Phi_{diag}^\delta}
\def\wUdiag{{\widehat\Phi}_{diag}^\delta}
\def\wUch{\widehat{\Phi}_{ch}}
\def\Umain{\Phi^{\alpha, \beta, \delta, \eta}}
\newcommand\Up[1]{\Phi_{#1}}

\def\Lx{L^{\alpha, \beta}_{x}}

\def\ddt{\frac{d}{dt}}

\def\Snm1{\mathbb S^{n-1}}

\def\R{\mathbb R}

\def\C{\mathbb C}

\def\Zn{\mathbb Z_N}

\def\a{\alpha}
\def\b{\beta}
\def\gam{\gamma}

\def\grad{\nabla}

\newcommand\textem[1]{{\em #1}}

\begin{document}

\title{Frame potentials and the geometry of frames}
\thanks{This work was partially supported by NSF DMS-1109545
and by the Alexander von Humboldt 
Foundation.
}

\date{} 
\author{Bernhard G. Bodmann}
\address{651 Philip G. Hoffman Hall, Department of Mathematics, University of Houston, Houston, TX 77204-3008}

\author{John Haas}
\address{651 Philip G. Hoffman Hall, Department of Mathematics, University of Houston, Houston, TX 77204-3008}

\begin{abstract}
\jh{
This paper concerns the geometric structure of 
optimizers for frame potentials. We consider
finite, real or complex frames and rotation or unitarily invariant potentials,
and mostly specialize to Parseval frames,
meaning the frame potential to be optimized is a function
on the manifold of Gram matrices belonging
to finite Parseval frames.
Next to the known classes of equal-norm and equiangular Parseval frames, 
we introduce equidistributed Parseval frames, which are more general than the equiangular type but have more structure than equal-norm ones. We also provide examples where
this class coincides with that of Grassmannian frames,
the minimizers for the maximal magnitude among inner products between
frame vectors.}
These different types of frames are characterized in relation to the optimization
of frame potentials.
Based on results by {\L}ojasiewicz,
we show that the gradient descent for a real analytic frame potential
on the manifold of Gram matrices belonging to Parseval frames
always converges to a critical point. We then
derive geometric structures associated with the critical points of different choices of
frame potentials. The optimal frames for families of such potentials 
are thus shown to be equal-norm, or additionally equipartitioned, or even equidistributed.

\end{abstract}

\maketitle

\section{Introduction}

The frame-based expansion of vectors in Hilbert spaces has
become an increasingly popular tool in many areas of mathematics \cite{FINITEFRAMESBOOK},
science and engineering \cite{KOVACEVICCHEBIRA}. Frames have many properties
comparable to orthonormal bases, but are not required to form linearly independent sets, and therefore offer more flexibility to accommodate specific design requirements. 
Since the early days of frame theory \cite{DuffinSchaffer}, structured frames 
have been given special consideration. The structure can be of an algebraic
nature, for example when the frame is constructed with the help of 
group representations \cite{DuffinSchaffer,HANLARSON,Chr03,Heil11}, or it can be present in the form of geometric conditions on the frame vectors. 
Such geometric conditions often result from frame design problems, for example
when
frames are used as analog codes for erasures \cite{CasKov03,HP04,BodPau05,Kal06}, 
or for beam forming in electrical engineering \cite{Heath04,HEATH},
or for quantum state tomography \cite{Zau, RBKSC04}. Equiangular tight frames are optimal for many of these applications. Such frames 
are most similar in character to orthonormal bases in that they
provide simple expansions for vectors, the norms of all the frame vectors are identical and the 
inner products between any two frame vectors have the same magnitude. 
The optimality can be expressed conveniently in terms of a so-called frame potential.

\jh{It is appealing that optimization principles have such a simple geometric
consequence. The characterization of equiangular tight frames was the motivation
to implement numerical searches for equiangular Parseval frames via the
minimization of frame potentials. However, it is known that
equiangular tight frames do not exist for all numbers of frame vectors and dimensions of the Hilbert space \cite{STDH07, welch:bound}.  } Thus, one is left with unanswered questions: Is there a more general structure for Parseval frames that includes equiangular ones as a special case and characterizes
optimizers for certain  frame potentials?
Moreover, is there a program for the optimization of these frame potentials
which is guaranteed to converge?
The application-oriented optimality principles lead from the special case of 
equiangular tight frames to the class
of Grassmannian frames. These minimize the maximum inner product 
among frame vectors subject to certain constraints \cite{HS03}. 
However, there was no clear indication whether these frames
exhibit special structures that might help with their design \cite{HP04}.

The present work was driven by the question whether there is a larger class
of structured frames that includes the equiangular tight case, 
which provides more examples
of optimal frames. In order to study a large class of optimization problems, we have
considered many types of frame potentials. Frame design by optimization of such
potentials is often done by a gradient descent procedure, so the classification
of critical points for the frame potentials is fundamental for the understanding of optimizers.

The results in this paper concern the relationship between the choice of frame potentials, real analytic functions  on the manifold 
of Gram matrices belonging to Parseval frames, and the geometric character of optimizers. 
The structure of this paper is as follows:
We first review and slightly extend results that characterize
the structure of frames which meet certain bounds for frame potentials.

Apart from well-studied types of frames, such as equal-norm Parseval frames and
equiangular Parseval frames, we introduce a new variant, which we call equidistributed Parseval frames.
We demonstrate that this type exhibits many examples for any finite number of frame vectors and the Hilbert spaces
they span. We also provide examples where this type of structure coincides with that of Grassmannian frames, the
minimizers for the maximal magnitude appearing among all inner products between pairs of frame vectors.

We then show that a result of Lojasciewicz implies that, for
any real analytic frame potential, the gradient descent converges 
to a critical point. In addition, the
sublevel sets accessible to the gradient descent with a sufficiently small
initial potential rules out the convergence of the gradient descent to 
the orthodecomposable case.

What remains is to classify critical points. 
Central to our main result is the development of a four parameter family of analytic frame potentials, $\{ \Umain = \Usum + \Udiag + \Uch \}_{\alpha, \beta, \delta, \eta \in (0,\infty)}$, each member of which is a sum of the nonnegative potentials $\Usum$, $\Udiag$, and $\Uch$ (see Definitions \ref{defUsum}, \ref{defUdiag}, \ref{defUch} and \ref{defUmain}).  Furthermore, each $\Umain$ is at most quadratic in the elementary one parameter potential function, $E_{x,y}^{\eta} $ (see Definition~\ref{defExy}).
The parameters $\a, \b,$ and $\delta$ induce weights that determine the proportionality of how the diagonal versus the off-diagonal entries of $G$ contribute to the potential value.  

Some results based on manifolds of equal-norm frames identify undesirable
critical points for frame potentials, so-called orthodecomposable frames \cite{CFM12,Strawn2012}, see also \cite{Strawn:mfdstruc}.
In Section~\ref{sectionusum}, we show that whenever the value of $\Usum(G)$ is sufficiently low, then $G$ cannot contain zero entries, thereby ruling out the orthodecomposable case.  
In Section~\ref{sectionudiag}, we see that whenever $G$ contains no zero entries and $\nabla \Udiag(G) = 0$ for all $\delta$ in a positive open interval, then $G$ is equal norm.  In Section~\ref{sectionuch}, we show that whenever $G$ contains no zero entries, it is equal-norm, and $\nabla \Uch(G) = 0$ for all $\a, \b$ in positive open intervals, then $G$ must be what we call equidistributed.  

Combining these results in Section~\ref{sectionumain} leads to 
a theorem which states that whenever the value $\Umain(G)$ is sufficiently low and  $\nabla \Umain(G) = 0$ for all $\a, \b, \delta$ in positive open intervals, then $G$ is equidistributed.  This is followed by Theorem~\ref{thmcharnozeros}, where we provide a characterization of equidistributed frames which do not exhibit orthogonality between any of the frame vectors.  

\jh{Finally, in Section~\ref{sectiongrass}, we show how a limiting procedure can give rise to Grassmannian frames with desirable structures.}

\section{Preliminaries}   

\subsection{Elementary properties of frames}

\begin{defn} \label{defframe}
A family of vectors $\mathcal{F}=\{ f_j \}_{j \in J}$ is a \textem{frame} for a real or complex Hilbert space $\mathcal{H}$ if there are constants $0 < A \leq B < \infty$ such that for all $x \in \mathcal{H}$,
$$
A \| x \|^2 \leq \sum\limits_{j \in J} | \langle x, f_j \rangle |^2 \leq B \| x \|^2 \, .
$$
We refer to the largest such $A$ and the smallest such $B$ as the \textem{lower} and \textem{upper frame bounds}, respectively.  In the case that $A=B$, we call $\mathcal{F}$ a \textem{tight frame}, and whenever $A=B=1$, then $\mathcal{F}$ is a \textem{Parseval frame}.  If $\|f_j\| = \| f_l \|$ for all $j,l \in J$, then $\mathcal{F}$ is an \textem{equal norm frame}.  If $\mathcal{F}$ is a an equal-norm frame and there exists a $C \ge 0$ such that  $| \langle f_j, f_l \rangle | = C$ for all $j, l \in J$ with $j \neq l$, then we say $\mathcal{F}$ is \textem{equiangular}.  The \textem{analysis operator} of the frame is the map $V: \mathcal{H} \rightarrow l^2 (J)$ given by $(Vx)_j = \langle x, f_j \rangle$.  Its adjoint, $V^*$, is the \textem{synthesis operator}, which maps $a \in l^2 (J)$ to $V^* (a) = \sum\limits_{j \in J} a_j f_j.$  The \textem{frame operator} is the positive, self-adjoint invertible operator $S = V^* V$ on $\mathcal{H}$ and the \textem{Gramian} is the operator $G=V V^*$ on $\ell^2(J)$. 
\end{defn}

In this paper, we focus on the case that $\mathcal{H} = \mathbb F^K$,
where $\mathbb F= \mathbb C$ or $\mathbb R$, $K$ is a positive integer, and always choose the canonical sesquilinear
inner product.   Thus, $K$ always denotes the dimension of $\mathcal{H}$ over the field $\mathbb F$.  Furthermore, we restrict ourselves to finite frames indexed by $J = \Zn$, where $N\ge K$, and reserve the letter $N$ to refer to the number of frame vectors in the frame(s) under consideration.  When the group structure of ${\mathbb Z}_N$ is not important, we also number the frame vectors with $\{1,2,\dots, N\}$, with the tacit understanding that $N \equiv 0 \,(\mbox{mod}\, N)$. 
Since this paper is mostly concerned with finite Parseval frames, we
call a Parseval frame for $\mathbb F^{K}$ consisting of $N$ vectors an 
\textem{$(N,K)$-frame}. 

If $\mathcal{F}$ is an $(N,K)$-frame, then $V$ is an isometry, since
$
\| Vx \|_2^2
=
\sum_{j =1}^N | \langle x, f_j \rangle|^2
= \| x \|^2
$
holds for all $x \in {\mathbb F}^{K}$. Hence, 
we obtain the reconstruction identity
$x=\sum_{j=1}^N \langle x, f_j \rangle f_j$, or in terms of the analysis and synthesis operators,
$x=V^*Vx$.
In this case, we also have that the the Gramian $G = VV^*$ is a rank-$K$ orthogonal projection,
because $G^* G=V V^* V V^*= V V^* = G$ and the rank of $G$ equals the trace, $tr(G) = K$. 

Many geometric
properties of frames discussed in this paper only depend on the inner products between frame vectors
and on their norms, which are collected in the Gramian. This means, most of the results presented hereafter
refer to equivalence classes of frames.

\begin{defn} Two frames $\mathcal F = \{f_j \}_{j \in J}$ and
$\mathcal F' = \{f'_j\}_{j \in J}$ for a real or complex Hilbert space $\mathcal H$
are called \emph{unitarily equivalent} if there exists an orthogonal or unitary operator $U$ on $\mathcal H$, respectively,
such that $f_j = U f'_j$ for all $j \in J$.
\end{defn}
 
Each equivalence class of frames is characterized 
by the corresponding Gram matrix.

\begin{prop}
The Gramians of two frames $\mathcal F = \{f_j \}_{j \in J}$ and
$\mathcal F' = \{f'_j\}_{j \in J}$ for a finite dimensional real or complex Hilbert space $\mathcal H$
are identical if and only if the frames are unitarily equivalent. 
\end{prop}
\begin{proof}
Assuming $G$ is the Gramian for the frame $\mathcal F$ as well 
as for the frame $\mathcal F'$, then $G=VV^*=V'(V')^*$, where
$V$ and $V'$ are the analysis operators belonging to $\mathcal F$
and $\mathcal F'$, respectively. By the polar decomposition,
$V=(V V^*)^{1/2} U = G^{1/2} U$ and $V'=(V' (V')^*)^{1/2} U'=G^{1/2} U'$
with isometries $U$ and $U'$ from $\mathcal H$ to $\ell^2(J)$, thus 
$V^*= U^* U'(V')^*$. By the frame property, the range of $U$ is identical to that of $U'$
and that of $G$, so $Q=U^* U'$ is unitary, which shows that $V^* e_j= Q (V')^* e_j$
for each canonical basis vector $e_j$ in $\ell^2(J)$, or equivalently, $f_j = Q f'_j $
for all $j \in J$.
Conversely, if $\mathcal F$ and $\mathcal F'$ are unitarily equivalent,
then it follows directly that the Gramians of both frames are identical. 
\end{proof}

 Special
emphasis is given to the Gram matrices of Parseval frames.
By the spectral theorem and the condition
$G^2=G$ this set is precisely the set of
rank-$K$ orthogonal projections, the
Gram matrices 
of Parseval frames for $\mathbb F^K$.

\begin{defn}
We define for $\mathbb F = \mathbb R$ or $\mathbb C$
$$\M = \{ G \in \mathbb F^{N \times N}: 
G = G^2= G^*, tr(G)=K \} \, .$$
\end{defn}

This subset of the $N \times N$ Hermitians is, in fact, a real analytic submanifold (see Appendix~\ref{app:Mrealanalytic} for a proof of this statement).

\subsection{Equidistributed frames}

A type of frame that emerged in our study of frame potentials is what we call equidistributed.
These frames include many structured frames that have already appeared in the literature: equiangular Parseval frames,
mutually unbiased bases, and group frames.

\begin{defn}\label{defequidist}
Let $\mathcal{F} = \{ f_j \}_{j=1}^N$ be an $(N, K)$-frame and let $G$ be its Gramian. 
The frame $\mathcal{F}$ is called \emph{equidistributed} if for each pair
$p, q \in {\mathbb Z}_N$, there exists a permutation $\pi$ on ${\mathbb Z}_N$
such that $|G_{j,p} | = | G_{\pi(j),q} |$ for all $j \in {\mathbb Z}_N$.
In this case, we also say that $G$ is equidistributed.
\end{defn}

In other words, $\mathcal{F}$ is equidistributed  if and only if 
the magnitudes in any column of the Gram matrix repeat in any other column,
up to a permutation of their position.
For Parseval frames, equidistribution implies that all frame vectors have the same norm.

\begin{prop}
If $\mathcal{F}$ is an equidistributed $(N,K)$-frame, then
$\|f_j\|^2 = K/N$ for each $j \in {\mathbb Z}_N$.
\end{prop}
\begin{proof}
By assumption, for each $p \in {\mathbb Z}_N$ there exists $\pi$
such that $|G_{j,p}| = |G_{\pi(j),1}|$ holds for the entries of the 
associated Gram matrix $G$
for all $j \in {\mathbb Z}_N$ and thus by the Parseval identity
$$
   \|f_p\|^2 = 
      \sum_{j=1}^N |\langle f_p, f_j \rangle|^2 = 
   \sum_{j=1}^N |G_{j,p} |^2
   = \sum_{j=1}^N |G_{\pi(j),1}|^2 = \|f_1\|^2 \, . 
$$
The trace condition $\sum_{j=1}^N G_{j,j} = \sum_{j=1}^N \|f_j\|^2 = K$
for the Gram matrices of Parseval frames
then implies that each vector has the claimed norm.   
\end{proof}

Below are a few examples to illustrate our definition.
To begin with, any equiangular Parseval frame is equidistributed.

\begin{ex}\label{exequiangular}
 Equiangular Parseval frames.
Let $G$ be the Gram matrix of an equiangular $(N,K)$-frame.  Since the magnitudes of the entries of any column of $G$ consist of $N-1$ instances of $C_{N,K}$ and one instance of $\frac{K}{N}$, $G$ is equidistributed.  \end{ex}

A class of frames with close similarities to equiangular Parseval frames is called Mutually Unbiased Bases \cite{MUBs-Schwing, MUBs-Ivan}. We show that
a slightly more
general class, \textem{Mutually Unbiased Basic Sequences}, is equidistributed. In this case, the frame vectors
are a collection of orthonormal sequences that are mutually unbiased. To include Parseval frames, we allow for an overall
rescaling of the norms.

\begin{ex}  Mutually Unbiased Basic Sequences. \label{MUBS}\label{ex46hankel} 
Let $N=ML$ and $G$ be such that the matrix $Q$ whose entries are $Q_{j,l} = |G_{j,l}|$
is the sum of Kronecker products of the form $Q = b I_M \otimes I_L + c (J_M-I_M)\otimes J_L$,
where $b>0$, $c \ge 0$, and the matrices $I_M$ and $I_L$ are the $M\times M$ and $L \times L$ identity matrices, and
$J_M$ and $J_L$ are the matrices of corresponding size whose entries are all $1$.  Each row of $G$ 
has one entry of magnitude $b$, $L-1$ vanishing entries and $(M-1)L$ entries of magnitude $c$, so $G$ 
is equidistributed. We also provide a concrete nontrivial example of such a $(6,4)$-frame
with $M=3$ and $L=2$.

Let $\omega = e^{2 \pi i/8}$, a primitive $8$-th root of unity, 
$\lambda=\sqrt{\frac{1}{18}}$ and let 
$$
G
=
\left(
\begin{array}{cccccc}
\frac{2}{3} & 0 & \lambda & i \lambda & \lambda & \lambda \\
 0               & \frac{2}{3} & i \lambda & \lambda & -\lambda & \lambda \\
 \lambda & -i \lambda & \frac{2}{3} & 0 & \lambda \omega^5 & \lambda \omega^3 \\
 -i \lambda & \lambda & 0 & \frac{2}{3}  & \lambda \omega & \lambda \omega^3 \\
 \lambda & -\lambda & \lambda \omega^3 & \lambda \omega^7 & \frac{2}{3} & 0 \\
 \lambda & \lambda & \lambda \omega^5 & \lambda \omega^5 & 0 & \frac{2}{3}
\end{array}
\right).
$$
One can verify that $G=G^*=G^2$ and clearly $tr(G)=4$.  Thus, $G \in {\mathcal M}_{6,4}$.  Since the magnitudes of the entries of every column consist of one instance of $0$, one instance of $\frac{2}{3}$, and four instances of $\lambda$, it follows that $G$ is equidistributed.  \end{ex}

\begin{ex}\label{exharmonic} Group frames.
Let $\Gamma$ be a finite group of size $N=|\Gamma|$ and $\pi: \Gamma \to B(\mathcal H)$ be an orthogonal or unitary representation of $\Gamma$ 
on the real or complex  
$K$-dimensional Hilbert space $\mathcal H$, respectively. Consider the orbit $\mathcal F = \{f_g = \pi(g) f_e\}_{g \in \Gamma}$
generated by a vector $f_e$ of norm $\|f_e\| = \sqrt{K/N}$, indexed by the unit $e$ of the group. If
$\mathcal F $ is a Parseval frame $\mathcal F = \{f_g\}_{g \in \Gamma}$, then
$\mathcal F$ is equidistributed, because $\langle f_g , f_h\rangle = \langle \pi(h^{-1}g) f_e, f_e\rangle$ and
left multiplication by $h^{-1}$ acts as a permutation on the group elements. 

To have the Parseval property, it is sufficient if the representation is irreducible, but there are also examples where this is not the case,
such as the harmonic frames \cite{GVT98} which are obtained with the 
representation of the abelian group $({\mathbb Z}_N,+)$ 
on $\mathcal H$.


%

\end{ex}

\begin{ex} Tensor Products of Equidistributed Frames.  
Let $1 \leq K_1 < N_1$ and $1 \leq K_2 < N_2$ be integers, let $G_1 \in {\mathcal M}_{N_1, K_1}$ and $G_2 \in {\mathcal M}_{N_2, K_2}$ be equidistributed, and consider the Kronecker product $G = G_1 \otimes G_2$.  Then $G$ is an $N_1 N_2 \times N_1 N_2$ Hermitian matrix such that $G^2 = (G_1 \otimes G_2)^2 = G_1^2 \otimes G_2^2 = G_1 \otimes G_2 =G$, so it is an orthogonal projection.  Furthermore, $G_{j,j} = \frac{K_1 K_2}{N_1 N_2}$ for all $j \in {\mathbb Z}_{N_1 N_2}$, so $tr(G) = K_1 K_2$.  Therefore, $G \in {\mathcal M}_{N_1 N_2, K_1 K_2}$.  Now let $p, q \in {\mathbb Z}_{N_1 N_2}$ 
with $p=p_1 N_1 + p_2$ and $q=q_1 N_1 + q_2$, and
let $Q, Q_1,$ and $Q_2$ denote the matrices whose entries are the absolute values of the entries of $G, G_1,$ and $G_2$, respectively.  Since $G_1$ and $G_2$ are equidistributed, row $p$ of $Q$ is of the form
$$\rho_p = ( \begin{array}{cccc} (Q_1)_{p_1, 1} X   &  (Q_1)_{p_1, 2} X & \cdots & (Q_1)_{p_1, N_1}  X \end{array})$$
where $X$ is row $p_2$ of $Q_2$ and row $q$ of $Q$ is of the form 
$$\rho_q = ( \begin{array}{cccc} (Q_1)_{q_1, 1} Y  &  (Q_1)_{q_1, 2} Y  & \cdots & (Q_1)_{q_1, N_1} Y \end{array}),$$
where $Y$ is row $q_2$ of $Q_2$. Since $G_1$ and $G_2$ are equidistributed, there is $\pi_1$ such that
$|(Q_1)_{q_1,j}| = |(Q_1)_{q_2, \pi_1(j)}|$ for each $j \in {\mathbb Z}_{N_1}$ and similarly, the magnitudes of the entries in $Y$ are obtained from
those  in $X$ by applying a permutation $\pi_2$ to the indices. Thus, the 
magnitudes of the entries of $\rho_q$ are a permutation of those of $\rho_p$, so $G$ is equidisributed.
\end{ex}

\subsection{Grassmannian \jh{Parseval} frames and equidistribution}

It is known that equiangular Parseval frames do not exist for all choices of $K$ and $N \ge K$.
In the absence of such frames, perhaps the best alternative is known as Grassmannian frames \cite{HS03}.
These minimize the maximal magnitude of the inner products between any two frame vectors, subject to certain constraints, for example
among equal-norm frames. Here, we consider such minimizers among the family of Parseval frames. As before, we express this
property of frames in terms of the corresponding Gram matrices.

\begin{defn}
Let $G$ be the Gram matrix for any frame consisting of $N$ vectors over $\mathbb F^K$ and 
let $\mu(G) = \max_{j \ne l} |G_{j,l} |$. 
A frame $\mathcal{F}$ is called a \textem{Grassmannian Parseval frame} if 
it is an $(N,K)$-frame and if its Gram matrix $G$ satisfies
$$
   \mu (G) = \min_{G' \in \M} \mu(G') \, .
$$
\end{defn}

Since the space of rank-$K$ orthogonal projections in ${\mathbb F}^{N\times N}$ is compact, the minimum on the right hand side
exists by the continuity of $\mu$, and thus Grassmannian Parseval frames exist for any $N$ and $K$.

In the usual topology of $\mathbb F^{N \times N}$, the Gram matrices belonging to equal-norm frames whose vectors have norm $\sqrt{K/N}$ 
form a paracompact set.
Moreover, the subset of Gram matrices belonging to equal-norm $(N,K)$-frames is compact and non-empty
for each $K$ and $N \ge K$.
By the continuity of $\mu$ and the compactness, 
minimizers for $\mu$ always exist over this restricted space.  We call such minimizers Grassmannian equal-norm Parseval frames.

\begin{defn}
Let $\Omega_{N,K}$ denote the set of Gram matrices corresponding to equal-norm frames $\mathcal F = \{f_j \}_{j=1}^N$ for $\mathbb F^K$
with $\|f_j \|^2 = K/N$ for all $j \in \Zn$.  
A frame $\mathcal{F}$ is called a \textem{Grassmannian equal-norm frame} 
for $\mathbb F^K$ if its Gram matrix $G$ is in  $\Omega_{N,K}$ and 
$$
   \mu (G) = \min_{G' \in \Omega_{N,K}} \mu(G') \, .
$$

A frame $\mathcal{F}$ is called a \textem{Grassmannian equal-norm Parseval frame} if 
it is an equal-norm $(N,K)$-frame and if its Gram matrix $G$ satisfies
$$
   \mu (G) = \min_{G' \in \M\cap\Omega_{N,K}} \mu(G') \, .
$$
\end{defn}
By the set inclusion $\M\cap\Omega_{N,K}\subset \M$, a Grassmannian Parseval frame which is equal norm is a Grassmannian equal-norm Parseval frame.  Similarly,  by $\M\cap\Omega_{N,K}\subset \Omega_{N,K}$, a Grassmannian equal-norm frame which is Parseval is also a Grassmannian equal-norm Parseval frame.

\jh{In \cite{HP04}, Grassmannian equal-norm Parseval frames are shown to
be the optimal frames when frames are used as analog codes and up to two frame coefficients are erased in the course of a transmission. Based on the numerical construction of optimal frames for $\mathbb R^3$, they did not seem to have
a simple geometric structure, apart from the case of equiangular Parseval frames.
Nevertheless, it is intriguing that there are other dimensions 
for which we can find equal-norm Parseval frames that are not equiangular, but equidistributed.}
We provide examples for the case where $\mathbb F = \mathbb R$. 

\begin{ex}\label{Ex:GrassmannianSemicirc}
Let $K=2$ and $N>3$.  For $j \in \Zn$, let 
$$f_j = \sqrt{\frac{2}{N}}
 \left( 
 \begin{array}{cc}
 \cos (\pi j/N) \\
 \sin(\pi j/N)
 \end{array}
 \right) ,
 $$ 
 then $\mathcal{F}=\{ f_j : j \in \Zn \}$ is easily verified to be a Parseval and equidistributed frame, but it is not equiangular.  Furthermore, as shown in \cite{BK06}, this frame is a minimizer of $\mu$ over the space of equal norm frames, so it must also be a minimizer of $\mu$ over the intersection of the equal-norm and Parseval frames.   Therefore, $\mathcal{F}$ is a Grassmannian equal-norm Parseval frame which is equidistributed.

\end{ex}

\begin{ex}\label{Ex:GrassmannianMUB}
Let $K=4$ and $N=12$.  Consider the $(12, 4)$-frame $\mathcal F$ with analysis operator $V$ whose vectors are given by the columns of the following synthesis matrix,
$$
V^*
=
\left(
\begin{array}{cccccccccccc}
\sqrt{\frac{1}{3}} & 0 & 0 & 0     & a & a & a & -a        & a & a & a & -a \\
0 &\sqrt{\frac{1}{3}}  & 0 & 0     & a & a & -a & a        & a & -a & a & a \\
0& 0 & \sqrt{\frac{1}{3}} & 0     & a & -a & -a & -a     & a & -a & -a & -a \\
0 & 0 & 0 & \sqrt{\frac{1}{3}}     & a & -a & a & a        & -a & -a & a & -a \\
\end{array}
\right),
$$
where $a=\sqrt{1/12}$.  This is an equal-norm sequence of vectors which can be grouped into 3 sets of 4 orthogonal vectors, thus it is straightforward to verify that this is a Parseval frame for $\mathbb R^4$. In addition, inspecting inner products between the vectors shows that they form, up to an overall scaling of the norms, mutually unbiased bases. 
Thus $\mathcal F$ is equidistributed (and equal-norm).  To see that this is a Grassmannian equal-norm Parseval frame, we note that showing that a frame to be Grassmannian equal-norm is equivalent to showing that it corresponds to an optimal sphere packing; that is, we desire the absolute value of the smallest angle to be as large as possible.  With this in mind, we compute that the absolute values of the sines of all possible angles between frame vectors belong to the set $\{1, \sqrt{3}/2 \}$. The orthoplex bound \jh{(see \cite{ORTHOPLEX} for details)} shows us that $\sqrt{3}/2$ is indeed the largest possible value that the sine of the smallest angle in such a frame can take, thereby verifying that this is a Grassmannian equal-norm Parseval frame. 
\end{ex}

\section{Bounds for frame potentials and structured frames}

Special classes of frames are characterized with the help of inequalities
for frame potentials, which relate to Frobenius norms.
This is the case for equal-norm Parseval frames and
for equiangular Parseval frames.

\begin{defn}
The $p$-th frame potential of a frame $\mathcal F = \{f_j\}_{j=1}^N$
for a real or complex Hilbert space $\mathcal H$ is given by
$$
  \Up{p}(\mathcal F) = \sum_{j,l=1}^N |\langle f_j , f_l \rangle |^{2p}
$$ 
\end{defn}

Benedetto and Fickus showed that among frames $\{f_j\}_{j=1}^N$ 
whose vectors all have unit norm, 
the tight frames are minimizers for $\Up{1}$ \cite{BF03}.
We adjust the norms to obtain a characterization of
equal-norm Parseval frames.

\begin{thm}[Benedetto and Fickus \cite{BF03}]
Let $\mathcal F =\{f_j\}_{j=1}^N$ be a frame for $\mathbb F^K$, with $\mathbb F = \mathbb R$
or $\mathbb C$, let $\|f_j\|^2= K/N$ for each $j \in {\mathbb Z}_N$,
then
$$
   \Up{1}(\mathcal F) = \sum_{j,l=1}^N |\langle f_j, f_l \rangle |^2 \ge K
$$
   and equality holds if and only if $\mathcal F$ is Parseval.
\end{thm}
\begin{proof}
The assumption on the norms is equivalent to the condition 
on the diagonal entries,
$G_{j,j}=K/N$,
of the Gram matrix $G=V V^*$ of the frame $\mathcal F$.
By the Cauchy-Schwarz inequality with respect to the Hilbert-Schmidt inner product, $\Up{1}({\mathcal F}) = tr(G^2) \ge (tr(G P))^2/tr(P^2)$, where $P$ is the orthogonal
projection onto the range of $G$ in $\ell^2({\mathbb Z}_N)$.
However $tr(G P) = tr(G) = K = tr(P)=tr(P^2)$, thus the claimed
lower bound follows. The case of equality holds if and only if $G$ and
$P$ are collinear, which means whenever $\mathcal F$ is Parseval. 
\end{proof}

In analogy with the characterization of tight frames, if equiangular tight frames exist among unit-norm frames, 
then they are minimizers  for $\Up{p}$ if $p>1$ \cite{RBKSC04,Okt07}, see also \cite{welch:bound}. Again, we present this result with
rescaled norms to replace tight frames by Parseval frames.

\begin{thm}[]
Let $\mathcal F = \{f_j\}_{j=1}^N$ be a frame for $\mathbb F^K$, with $\mathbb F = \mathbb R$
or $\mathbb C$, and $\|f_j\|^2= K/N$ for all $j \in {\mathbb Z}_N$, and let $p>1$,
then
$$
  \Up{p}(\mathcal F) = \sum_{j,l=1}^N |\langle f_j, f_l \rangle |^{2p}
  \ge \frac{K^{2p}(N-1)^{p-1} + K^p  (N-K)^p}{(N-1)^{p-1} N^{2p-1}}
$$
and equality holds if and only if $\mathcal F$ is an equiangular Parseval frame.
\end{thm} 
\begin{proof}
With the elementary properties
of equal-norm frames and Jensen's inequality, we obtain the bound
$$
  \Up{p}(\mathcal F) = \sum_{j=1}^N \|f_j\|^{4p} + \sum_{j \ne l} |\langle f_j, f_l \rangle|^{2p}
  \ge \frac{K^{2p}}{N^{2p-1}}
  + \frac{1}{N^{p-1}(N-1)^{p-1}} (\sum_{j\ne l} |\langle f_j, f_l \rangle|^2)^p \, .
  $$
  Expressing this in terms  of $\Up{1}$ and using the preceding theorem then gives 
\begin{align*}
 \Up{p}(\mathcal F)  
  & \ge \frac{K^{2p}}{N^{2p-1}} + \frac{1}{N^{p-1}(N-1)^{p-1}} (\Up{1}(\mathcal F) - \frac{K^2}{N})^{p}\\
   & \ge \frac{K^{2p}}{N^{2p-1}} + \frac{1}{N^{p-1}(N-1)^{p-1}} \frac{ K^p(N-K)^p}{N^p}\\
   & = \frac{K^{2p}(N-1)^{p-1} + K^p  (N-K)^p}{(N-1)^{p-1} N^{2p-1}} \, .
\end{align*}
  Moreover, equality holds in the Cauchy-Schwarz and Jensen inequalities
  if and only if $\mathcal F$ is Parseval and if there is $C\ge 0$
  such that $|\langle f_j , f_l \rangle | = C$ for all $j, l \in {\mathbb Z}_N$
  with $j \ne l$.
\end{proof}

If equality holds, then inspecting the proof shows that 
the magnitude of the 
off-diagonal entries of the Gram  matrix 
is a constant, see also \cite{GR09}, \cite{HP04}, and \cite{HS03}, which we 
record for further use,
$$ 
   C_{N,K} = \sqrt{\frac{1}{N(N-1)} \Bigl(K - \frac{K^2}{N}\Bigr)} = \sqrt{\frac{K(N-K)}{N^2(N-1)}} \, .
$$

\begin{cor} Let $p>1$.
If $\mathcal F = \{f_j\}_{j=1}^N$ is a frame for $\mathbb F^K$
with
$ \|f_j\|^2 = K/N$ for each $j \in {\mathbb Z}_N$
and
$\Up{p}(\mathcal F)$ achieves the lower bound in the preceding theorem
 then 
 $\mathcal F$  is Parseval and
$|\langle f_j , f_l \rangle | = C_{N,K}$ for all $j \ne l$.
\end{cor}

By definition the value of $\Up{p}(\mathcal F)$
only depends on the entries of the corresponding Gram matrix.
Thus, the characterizations of equal-norm $(N,K)$-frames and of
equiangular $(N,K)$-frames are implicitly statements about equivalence
classes of frames.

\begin{cor}
If two frames $\mathcal F = \{f_j \}_{j =1}^N$ and
$\mathcal F' = \{f'_j\}_{j =1}^N$ are unitarily equivalent, then
$\Up{p}({\mathcal F}) = \Up{p}({\mathcal F}')$.
\end{cor}

For this reason, we consider instead of $\Up{p}$ the corresponding function
of Gram matrices.

Moreover, it is for our purposes advantageous to consider the compact manifold
$\M$ consisting of Parseval frames instead of the open manifold of equal-norm frames. In this setting, we have analogous theorems which characterize 
the equal-norm case and the equiangular case.

\begin{thm}
Let $G \in \M$, then
$$
  \sum_{j=1}^N |G_{j,j} |^2 \ge \frac{K^2}{N}
$$
and equality holds if and only if $G_{j,j}=\frac{K}{N}$ for each $j \in \mathbb Z_N$.
\end{thm}
\begin{proof}
We know that $\sum_{j=1}^N G_{j,j} = K$, so the Cauchy-Schwarz inequality
gives
$$
   \sum_{j=1}^N |G_{j,j} |^2 \ge \frac{1}{N}(\sum_{j=1}^N G_{j,j})^2 = K^2/N
$$
and equality is achieved if and only if $G_{j,j}=G_{l,l}$ for all 
$j,l \in \mathbb Z_N$. By summing the diagonal entries of $G$, we then obtain
$N G_{j,j} = K$ for each $j \in {\mathbb Z}_N$.
\end{proof}

In terms of $(N,K)$-frames $\{f_j\}_{j=1}^N$, 
the function estimated here is $\sum_{j=1}^N \|f_j\|^4$.
Bodmann and Casazza called this a frame energy. They showed that 
if a Parseval frame has
a sufficiently small energy, then under certain additional conditions
an equal-norm  Parseval frame can be found in its vicinity \cite{BodmannCasazza:2010}.

Next, we state the characterization of equiangular Parseval frames.
Since this was only published in a thesis, we are grateful for the opportunity to 
present the proof here.

\begin{thm}\label{thmelwood}[\jh{Elwood \cite{Elw11}}]
Let $G \in \M$, then
$$
  \sum_{j,l=1}^N |G_{j,l} |^4 \ge \frac{K^2(K^2-2K+N)}{N^2(N-1)}
$$
and equality holds if and only if $G_{j,j} = K/N$ and $|G_{j,l}|=C_{N,K}$ for each $j\ne l$.
\end{thm}

\begin{proof}
We recall that by the fact that $G$ is an orthogonal rank-$K$ projection, one has that 
$\sum\limits_{j,l = 1}^N |G_{j,l}|^2 = \sum\limits_{j=1}^N G_{j,j} = K.$  With the help of these identities, we express the difference between the two  sides
of the inequality as a sum of quadratic expressions,
\begin{align*}
&\sum\limits_{j,l=1}^N  |G_{j,l}|^{4} -  \frac{K^2(K^2 + N -2K)}{N^2(N-1)}\\
&=
\tiny
{
\sum\limits_{
\begin{array}{cc}
j, l =1 \\
j  \neq l
\end{array}
}^N
}
\left(
|G_{j,l}|^2 - C_{N,K}^2
\right)^2
+
\sum\limits_{j=1}^N
\left(
G_{j,j}^2 - \frac{K^2}{N^2}
\right)^2
+
\frac{2K(K-1)}{N(N-1)}
\sum\limits_{j=1}^N
\left(
G_{j,j} - \frac{K}{N}
\right)^2 \, .
\end{align*}
In this form it is manifest that this quantity is non-negative and
that it vanishes if and only if $G$ is a rank-$K$ orthogonal projection
with $G_{j,j}=K/N$ for all $j$ and with $|G_{j,l}| = C_{N,K}$
for all $j \ne l$.
\end{proof}

It is natural to ask whether a characterization of Grassmannian Parseval frames in terms of frame potentials exists. 
In order to formulate this in a convenient manner, we first introduce another type of frame potential.

\begin{defn}\label{defExy} 
Let $G = (G_{a,b})_{a,b=1}^N  \in \M$.  Given $x,y  \in \Zn$ and $\eta >0$, we define  the \textem{exponential potential} $E_{x,y}^{\eta}: \M \rightarrow \R$ by
\begin{align*}
E_{x,y}^{\eta} (G)  
&= e^{\eta |G_{x,y}|^2} \, .
\end{align*}
Moreover, we define the \textem{off-diagonal sum potential} of $G$ as
$$
   \Phi_{\mathrm{od}}^\eta (G) = \sum_{j,l=1}^N (1-\delta_{j,l}) 
   E_{j,l}^{\eta}(G)  \, ,
$$
where the Kronecker symbol $\delta_{j,l}$ vanishes if $j \ne l$ and contributes
$\delta_{j,j}=1$ otherwise.
\end{defn}

Although for a fixed value of $\eta$ a Grassmannian Parseval frame may fail to be a minimizer for $\Phi_{\mathrm{od}}^\eta$,
the family of frame potentials $\{\Phi_{\mathrm{od}}^\eta\}_{\eta>0}$ characterizes them.

\begin{prop}
Let $G \in \M$,
then 
$$
  \mu(G) = \lim_{\eta \to \infty}\frac{1}{\eta} \ln(\Phi^\eta_{\mathrm{od}}(G)) \, .
$$
Moreover, if $G'$ belongs to
 a Grassmannian Parseval frame and $G'' \in \M$ does not, then there exists an $\eta>0$
such that $ \Phi_{\mathrm{od}}^{\eta'} (G') < \Phi_{\mathrm{od}}^{\eta'} (G'')$
for each $\eta'>\eta$.
\end{prop}

\begin{proof}
We have for any $G \in \M$
$$e^{\eta \mu(G)} \le \Phi_{\mathrm{od}}^\eta (G) \le N(N-1) e^{\eta \mu(G)}$$
thus
$\lim_{\eta \to \infty}\frac{1}{\eta} \ln(\Phi^\eta_{\mathrm{od}}(G)) = \mu(G)$.
Moreover, if $G'$ is the Gram matrix of a Grassmannian Parseval frame and $G''$ is not,
then $\mu(G'') = \mu(G')+\epsilon$ for some $\epsilon>0$ and if $\eta >\ln (N(N-1))/\epsilon$, then
$
  \eta \mu(G') + \ln N(N-1) < \eta \mu(G') + \eta \epsilon = \eta \mu(G'') \,  
$
and consequently $ \Phi_{\mathrm{od}}^\eta (G') \le N(N-1) e^{\eta \mu(G')} < e^{\eta \mu(G'')} \le \Phi_{\mathrm{od}}^\eta (G'')$.
\end{proof}

Although $\mu$ is continuous on $\M$, it is not globally differentiable. Thus,
locating even local minima is difficult. Fortunately, we can reduce
the minimization problem for $\mu$ to finding minimizers for
a sequence of frame potentials.

\begin{prop}\label{propodgivesgrass}
Let $\{\eta_m\}_{m=1}^\infty$ be a positive, increasing sequence such that $\lim\limits_{m\rightarrow \infty} \eta_m= + \infty$ and suppose $\left\{G(m)=\left(G(m)_{a,b} \right)_{a,b=1}^N \right\}_{m=1}^\infty \subseteq \M$ is a sequence such that $\Phi_{\mathrm{od}}^{\eta_m}$ achieves its absolute minimum at $G(m)$ for every $m \in \{1,2,3,...\}$, then there exists a subsequence $\left\{G(m_s)  \right\}_{s=1}^\infty$ and $G \in \M$ such that  $\lim\limits_{s\rightarrow \infty} G(m_s) = G$.  Furthermore, $G$ is the Gramian of a Grassmannian Parseval  frame.\end{prop}

\begin{proof}
By the compactness of $\M$,  there exists a subsequence $\left\{G(m_s)  \right\}_{s=1}^\infty$ and $G \in \M$ such that  $\lim\limits_{s\rightarrow \infty} G(m_s) = G$.  Now let $\widehat G \in\M$ correspond to any Grassmannian Parseval  frame.  By hypothesis, we have 
$$\Phi_{\mathrm{od}}^{\eta_{m_s}} \left( G(m_s) \right)
\leq
\Phi_{\mathrm{od}}^{\eta_{m_s}} \left( \widehat G \right)$$
for every $s \in \{1,2,3,...\}$.  Furthermore, for every $s \in \{1,2,3,...\}$, we have 
$$e^{\eta_{m_s} \mu(G(m_s))} \le \Phi_{\mathrm{od}}^{\eta_{m_s}} (G(m_s)) \le N(N-1) e^{\eta_{m_s} \mu(G(m_s))}.$$
Thus
$$\mu(G) = \lim\limits_{s\rightarrow \infty} \mu(G(m_s)) = \lim\limits_{s\rightarrow \infty} \frac{1}{\eta_{m_s}} \log\left(\Phi_{\mathrm{od}}^{\eta_{m_s}}(G(m_s))\right)
\leq
\lim\limits_{s\rightarrow \infty} \frac{1}{\eta_{m_s}} \log\left(\Phi_{\mathrm{od}}^{\eta_{m_s}}(\widehat G)\right) = \mu (\widehat G),
$$
where the first equality follows from continuity of the $\max$ function.  This shows that $G$ belongs to a Grassmannian Parseval frame.
\end{proof}

If the off-diagonal sum potential is properly complemented by terms for the diagonal entries of $G$, then
a simple characterization of equiangular Parseval frames can be derived.

\begin{defn}\label{defUsum} 
Let $G = (G_{a,b})_{a,b=1}^N  \in \M$.  Given  $\eta >0$, we define  the \textem{sum potential} of $G$ as
$$
   \Phi_{\mathrm{sum}}^\eta (G) = \Phi_{\mathrm{od}}^\eta (G) + 
   \sum_{j=1}^N e^{- \eta (K^2/N^2 - C_{N,K}^2)} 
   E_{j,j}^{\eta}(G)  \, .
$$
\end{defn}

\begin{prop}
Let $G \in \M$, then
$$
    \Phi_{\mathrm{sum}}^\eta (G) \ge N^2 e^{\eta(K/N^2 - K^2/N^3 + K(N-K))/N^3(N-1)} \, 
$$
and equality holds if and only if $G$ belongs to an equiangular Parseval frame.
\end{prop}
\begin{proof}
We use Jensen's inequality to obtain
$$
  \Phi_{\mathrm{sum}}^\eta (G) =   \sum_{j,l=1}^N e^{\eta |G_{j,l}|^2 - \delta _{j,l} (K^2/N^2 - C_{N,K}^2)} 
     \ge N^2 e^{(\eta/N^2) \sum_{j,l=1}^N (|G_{j,l}|^2 - \eta \delta _{j,l} (K^2/N^2 - C_{N,K}^2))} \, .
$$
Now using the Parseval property gives  $\sum_{j,l=1}^N |G_{j,l}|^2 = K $ and thus
$$
    \Phi_{\mathrm{sum}}^\eta (G) \ge N^2 e^{\eta(K/N^2 - K^2/N^3 + K(N-K))/N^3(N-1)} \, .
$$
Equality holds in Jensen's equality if and only if the average is over a constant. This implies that
the diagonal entries equal $G_{j,j}=K/N$ and the magnitude of the off-diagonal entries equals $C_{N,K}$.
\end{proof}

\jh{To conclude this section, we show that equidistributed frames can be characterized in terms of families of frame potentials based on exponential ones. To prepare this, we
introduce the notion of 
a frame being $\a$-equipartitioned.}

\begin{defn}\label{defequipart}
Let $\mathcal{F} = \{ f_j \}_{j=1}^N$ be an $(N, K)$-frame and fix $\a \in (0, \infty)$.  For any $x \in \Zn$, define $A_x^\a : = \sum\limits_{j \in \Zn} e^{\a |\langle f_j, f_x \rangle|^2}$.  If  $A_x^\a = A_y^\a$ for all $x, y \in \Zn$, then we say that $\mathcal{F}$ is \textem{$\a$-equipartitioned}. \end{defn}

\begin{prop}\label{propequipartfonopen}
Let $G= (G_{j,l})_{j,l=1}^N$ be the Gramian of an $(N,K)$-frame $\mathcal{F}$, and let $I \subseteq (0, \infty)$ be any open interval, then  $G$ is equidistributed if and only $G$ is $\alpha$-equipartitioned for all $\a \in I$.
\end{prop}
\begin{proof}
If $G$ is equidistributed, the magnitudes of every column are the same as those of any other column, up to permutation.  Thus, by definition of $\alpha$-equipartitioning, it is trivial to verify that then $G$ is $\alpha$-equipartitioned for all $\a \in I$.

Conversely, consider for each $x \in \Zn$ the function $f_x: (0, \infty) \rightarrow \R : \alpha \mapsto \sum\limits_{j \in \Zn} e^{\a | G_{x,j} |^2}$.  Let $x,y \in \Zn$ be arbitrary.  If $f_x(\a) = f_y(\a)$ for all $\a \in \R$ then  since $f_x$ and $f_y$ are both analytic functions which agree on an open interval, it follows by the principle of analytic continuation that they must agree on all of $(0, \infty)$.  In particular, this means that $\sum\limits_{j \in \Zn} e^{\a | G_{x,j} |^2} = \sum\limits_{j \in \Zn} e^{\a | G_{y,j} |^2}$ for all $\a \in (0, \infty)$.  Thus, 
$$
\lim\limits_{\a \rightarrow \infty} \frac{1}{\alpha} \log
\left(\sum\limits_{j \in \Zn} e^{\a | G_{x,j} |^2}\right)
= \lim\limits_{\a \rightarrow \infty}
\frac{1}{\alpha} \log \left(
 \sum\limits_{j \in \Zn} e^{\a | G_{y,j} |^2} \right) \, .
$$
  If the maximum magnitude in row $x$ is not equal to the maximum magnitude of row $y$, then this equation cannot hold.  Similarly, if these maximal magnitudes did not occur with the same multiplicity in each column, then again the equation would not be possible.  Thus, we can  remove the 
  index sets $M_x$ and $M_y$ corresponding to the maximal magnitudes in rows $x$ and $y$ 
from the  sum in the definition of $f_x$ and $f_y$   
  to obtain the new 
  identity
$$
   \sum\limits_{j \in \Zn\setminus M_x} e^{\a | G_{x,j} |^2} = \sum\limits_{j \in \Zn\setminus M_y} e^{\a | G_{y,j} |^2}
  $$ 
   for all $\alpha \in (0,\infty)$. Repeating the procedure of isolating the strongest growth rate
shows that every possible magnitude that appears in row $x$ must agree in multiplicity with every possible magnitude that appears in row $y$.  In other words, the magnitudes in row $x$ are just a permutation of those in row $y$.  Since $x$ and $y$ were arbitrary, we conclude that $G$ is equidistributed.
\end{proof}

\section{The Gradient Descent on $\M$}
In this section, we first show that following a gradient descent 
associated with a real analytic frame potential
always converges to a critical point.  This depends heavily on
 results by {\L}ojasiewicz.  To apply these results, we use that when $\mathbb F= \mathbb C$ (respectively $\mathbb F= \mathbb R$), the manifold $\M$ is embedded 
in the (linear) manifold of Hermitian (respectively symmetric) $N\times N$ matrices equipped with the Hilbert-Schmidt norm,
which induces a topology on $\M$ generated by the open balls $B(X,\sigma)=\{Y \in \M: \|Y-X\|<\sigma\}$
of radius $\sigma>0$ centered at each $X \in \M$. 
Moreover, the Hilbert-Schmidt norm induces a Riemannian structure
on the tangent space $T\M$.  Via the embedding, the tangent space 
$T_{G_0}\M$ at $G_0 \in \M$ is identified with a subspace of the Hermitian (respectively symmetric) matrices,
and the Riemannian metric is the real inner product $(X,Y) \mapsto X \cdot Y \equiv tr(XY)=tr(XY^*)$
restricted to the tangent space. 

We also recall that the gradient  of a differentiable function $F$ on $\M$ is the vector field $\nabla F$ 
which satisfies
the identity 
$$
    \nabla F(G_0) \cdot  X = \frac{d}{dt} \vert_{t=0} F(\gamma (t))
$$
for each $G_0 \in \M$ and each curve $\gamma \in C^1(\mathbb R, \M)$ with
$\gamma(0)=G_0$ and $\frac{d\gamma(0)}{dt} = X$. 

The frame potentials we
have defined on $\M$ are all given in terms of real analytic functions of matrix entries. 

\begin{thm}
(\cite{kurd:loj}, \cite{loj:ensanal}; see also \cite{merl:gradglows})
Let $\Omega$ be an open subset of $\mathbb R^d$ and
$F: \Omega \rightarrow \R$ real analytic. For any $x \in \Omega$
there exist $C,\sigma>0$ and $\theta \in (0,1/2]$ such that
for all $y \in B(x,\sigma)\cap \Omega$,
$$
   | F(y) - F(x)|^{1-\theta} \le C \| \nabla F(y) \| \, .
$$
\end{thm}

\begin{cor}\label{thmLojsat}
Let $\mathcal M$ be a $d$-dimensional real analytic manifold with a Riemannian structure.
Let $G_0 \in \Omega \subset M$ and let $W: \Omega \to \mathbb R$
be real analytic, then there exist
an open neighborhood $\mathcal U$ of $G_0$ in $\Omega$
 and constants $C>0$ and $\theta \in (0,1/2]$
 such that for all $G \in \mathcal U$,
 $$
   | W(G)-W(G_0) |^{1-\theta} \le C \|\nabla W(G) \| \, .
 $$
   \end{cor}
   \begin{proof}
Since the manifold is real analytic, after choosing a chart 
$\Gamma: \mathcal M \to \mathbb R^d$, there exists a
neighborhood $U$ of $x=\Gamma(G_0)$
in $\mathbb R^d$ such that 
$
  F= W \circ \Gamma^{-1}
$
  is a real analytic function on $U$.
  Thus, the {\L}ojasiewicz inequality gives
  a bound for the values of $F$
  in terms of the Euclidean gradient $\nabla F$
  in a set $B(x,\sigma) \cap U$.  
  However, $\Gamma$ is a diffeomorphism,
  thus by the continuity of the matrix-valued function obtained
  from applying the Riemannian metric to pairs 
  of the coordinate vector fields $\{\frac{\partial}{\partial x_j}\}_{j=1}^d$
  and by the fact that $B(x,\sigma)\cap U$ is paracompact in ${\mathbb R}^d$, 
  there exists $C'>0$ such that
  $\|\nabla F(\Gamma(G)\| \le C' \|\nabla W(G)\|$ if $\Gamma(G) \in B(x,\sigma)\cap U$.
  The combination of the {\L}ojasiewicz inequality in local coordinates
  with this norm inequality gives the claimed bound, valid in the neighborhood
  ${\mathcal U} = \Gamma^{-1}(B(x,\sigma)\cap U)$ of $G_0$.
\end{proof}

\subsection{Convergence of the gradient descent}

It is well known that the {\L}ojasiewicz inequality can be used to prove convergence of gradient flows induced by analytic cost functions on $\mathbb R^d$.  Here we provide a proof of convergence in our setting adapted from \cite{merl:gradglows}.

\begin{prop}\label{propconvergence}
Suppose that $W: \M \rightarrow \R$ is real analytic and let $\gamma$ be a global solution of the descent system $\dot{\gamma} = -\grad W(\gamma)$.  Then there is an element $G_0 \in \M$ such that $\gamma(t) \rightarrow G_0$ as $t \rightarrow\infty$ and $\grad W(G_0)=0$. 
\end{prop}

\begin{proof}
First, we observe that $W(\gam(t))$ is a nonincreasing function, since
\begin{align*}
\ddt W( \gam (t) ) &=  \grad W( \gam(t)) \cdot \dot{\gam}(t), \\
&=   - \grad W( \gam(t)) \cdot \grad W( \gam(t)) \\
&= - \| \grad W(\gam (t) )\|^2 \\
&\leq 0.
\end{align*}

Furthermore, since $\M$ is compact, there must some point $G_0 \in \M$ along with an increasing sequence $t_n$ in $\mathbb R$, $t_n \to \infty$, which satisfies that $\gam( t_n) \rightarrow G_0$.   Thus, the continuity of $W$ together with the fact that $t \mapsto W(\gamma(t))$ is nonincreasing
implies that $\lim\limits_{t \rightarrow \infty} W(\gam(t)) = W(G_0)$.

 Since adding a constant to our energy function will not alter the gradient flow, let us assume without loss of generality that $W(G_0) = 0$ and $W(\gam(t)) \geq 0$ for all $t \geq 0$.  

If $W(\gam(t)) = 0$ for some $t_0 \geq 0$, then it follows that $W(\gam(t)) = 0$ for all $t \geq t_0$.  In particular, since $\| \grad W (\gam (t)) \|^2 = - \ddt W( \gam (t)) =0$, we have $ \dot{\gam} (t) = \grad W (\gam(t)) =0$ for all $t \geq 0$.  In this case, the proof is complete.

        Henceforth, we will consider the case where $W(\gam (t)) > 0$ for all $t \geq 0$.  Due to Corollary~\ref{thmLojsat}, we know that since $W$ is real analytic in some neighborhood of $G_0$, it follows that there exist $C, \sigma > 0$ and $\theta \in (0, 1/2]$ such that
$$|W(\gam(t)) - W(G_0)|^{1 - \theta} = |W(\gam(t)) |^{1 - \theta}   \leq C \| \grad W (\gam (t)) \|$$
for all $t \geq 0$ where $\gam(t) \in B(G_0; \sigma) \cap \M$. Let $\epsilon \in (0, \sigma)$.  Then there exists a sufficiently large $t_0 \in \R_+$ that yields 
$$ \int\limits_{0}^{W(\gam(t_0))} \frac{C}{ s ^{1 - \theta}} ds  + \| \gam(t_0) - G_0 \| < \epsilon.$$
Setting $t_1 = \inf\{ t \geq t_0 : \| \gamma(t) - G_0 \| \geq \epsilon \},$ we note that the {\L}ojasiewicz inequality is satisfied for $t \in [t_0, t_1)$, which gives us
\begin{equation*}\label{eqLojas}- \ddt  \int\limits_{0}^{W(\gam(t))} \frac{C}{ s ^{1 - \theta}} ds 
= C \frac{- \ddt W(\gam(t)) }{|W(\gam(t))|^{1 - \theta}} 
= C \frac{\| \grad W(\gam(t))\|^2 }{|W(\gam(t))|^{1 - \theta}} \geq \\ \| \grad W(\gam(t)) \| = \| \dot{\gam} (t) \| . 
\end{equation*}
Since this inequality holds for any $t \in [t_0, t_1)$, it follows 
by integrating both sides
that for any $t \in [t_0, t_1]$ we have
\begin{align*}
\| \gam (t) - G_0 \| &\leq \| \gam (t) - \gam(t_0) \|  + \| \gam (t_0) - G_0 \| \\
&\leq \int\limits_{t_0}^{t_1} \| \dot{\gam} (s) \| ds + \| \gam (t_0) - G_0 \| \\
&\leq C \int\limits_{t_0}^{t_1} \frac{\| \grad W(\gam(s))\|^2 }{|W(\gam(s))|^{1 - \theta}} ds  + \| \gam (t_0) - G_0 \| \\
&=  - C\int\limits_{W(\gam(t_0))}^{W(\gam(t_1))} \frac{ dv}{v^{1 - \theta}}   + \| \gam (t_0) - G_0 \| \\
&\leq C \int\limits_{0}^{W(\gam(t_0))} \frac{ dv}{v^{1 - \theta}}   + \| \gam (t_0) - G_0 \| \\
&< \epsilon \, .
\end{align*}
This shows that $t_1 = + \infty$, so that 
\begin{align*}  \int\limits_{0}^{\infty} \| \dot{\gam}(t) \| dt   &\leq  C \int\limits_{0}^{\infty} \frac{\| \grad W(\gam(t))\|^2 }{|W(\gam(t))|^{1 - \theta}} dt  
= C \int\limits^{W(\gam(0))}_{0} \frac{dv}{v^{1- \theta}}  
< \infty \, .\end{align*}
Thus, we see that $\| \dot{\gam}(t) \| \in L^1 (\R_+)$, and conclude that $\gam(t) \rightarrow G_0$ as $t \rightarrow \infty$.
\end{proof}

\subsection{Characterization of fixed points for the gradient flow}

We recall that when $\mathbb F = \mathbb C$ (respectively $\mathbb F = \mathbb R$),  the embedding of $\M$ into the real vector space of Hermitian (respectively symmetric) $N\times N$ 
matrices induces a similar embedding of the tangent space to $\M$ at $G_0$,
$$
 T_{G_0}\M = \{ \dot\gamma(0): \gamma \in C^1(\mathbb R, \M), \gamma(0)=G_0 \}
 \subset {\mathbb F}^{N\times N}
$$
where $\dot\gamma$ is the (matrix-valued) derivative of $\gamma$. 
We use this embedding to compute gradients and characterize where the gradient vanishes.

\begin{lemma}\label{lemdimTG0} 
Let $G_0 \in \M$, then the real linear map
\begin{align*}
   P_{G_0}: {\mathbb F}^{N\times N} &\to {\mathbb F}^{N\times N}\\
              X &\mapsto (I-G_0)X G_0 +G_0 X^*(I-G_0)
\end{align*}
is the orthogonal projection onto $T_{G_0}\M$.
\end{lemma}

\begin{proof}
As a first step, we observe that because $P_{G_0}$ is idempotent, 
its  range 
is  the real vector space
$$
   \mathcal V_{G_0} = \{ X \in {\mathbb F}^{N \times N}: X=(I-G_0)X G_0 +G_0 X^*(I-G_0)\} \, .
$$   
We show that this vector space contains each tangent vector at $G_0$.
Let $\gamma: (a,b) \rightarrow \M$ be a smooth curve such that $0 \in (a,b)$ and $\gamma(0) = G_0$.  Since $\gamma(t)$ is an orthogonal projection for all $t \in (a,b)$, one has that $\gamma(t)^*=\gamma(t)$ and $\gamma(t) = \gamma(t )^2=\gamma(t)^3$ 
for all $t \in (a,b)$.  Therefore, differentiating 
$\gamma(t)^2 - \gamma(t)^3=0$ yields
$
   \gamma(t) \dot\gamma(t) \gamma(t) = 0 \, .
$
If $X = \dot\gamma(0)$, then at $t=0$ this gives
$$
  G_0 X G_0 = 0 \, .
$$   
Similarly, if $\iota(t)=I$, then the equations for the complementary projection, $\iota(t)-\gamma(t)=(\iota(t)-\gamma(t))^2=(\iota(t)-\gamma(t))^3$ result in the
identity
$$
  (I-G_0) X  (I-G_0) =0 \,
$$
for $X=\dot\gamma(0)$.
This, together with $\dot\gamma(0)^* = \dot\gamma(0)$ shows that each tangent vector is in ${\mathcal V}_{G_0}$.

Moreover, from Appendix~\ref{thmMisanalytic}, we know the dimension of $\M$ is $2K(N-K)$ when $\mathbb F= \C$ and $K(N-K)$ when $\mathbb F= \R$.  If $U$ is a unitary (respectively orthogonal) matrix whose columns are eigenvectors of $G_0$,
the first $K$ columns corresponding to eigenvalue one, then if $X=X^*$ and $X=(I-G_0)XG_0+G_0X(I-G_0)$, we know
$$
   X = U \left(\begin{array}{cc} 0 & Y\\ Y^* & 0 \end{array}\right) U^*
$$ 
with some $Y \in \mathbb F^{K \times N-K}$, so the real dimension of the space 
${\mathcal V}_{G_0}$ is $2K(N-K)$ when $\mathbb F= \C$ and $K(N-K)$ when $\mathbb F= \R$.
This is precisely the dimension of the real manifold $\M$, thus the vector space is the span of all the tangent vectors.

Finally, we note that the map $P_{G_0}$
is idempotent and self-adjoint with
respect to the (real) Hilbert-Schmidt inner product. Thus, it is an orthogonal projection onto its range, the tangent space of $\M$ at $G_0$.
\end{proof}

Since $P_{G_0}$ is the orthogonal projection onto $T_{G_0}\M$, it can be used to
construct Parseval frames for $T_{G_0} \M$ from suitable orthonormal sequences.
We first discuss the complex case and then the real case.
In the following, $\Delta_{a,b}$ with $a,b \in {\mathbb Z}_N$
denotes the matrix unit
whose only non-vanishing entry is a $1$ in the $a$th row and the $b$th column.

\begin{thm} \label{thm:ProjSandT}
Suppose $\mathbb F = \mathbb C$ and let $\{S_{a,a}: a \in {\mathbb Z}_N\} \cup \{S_{a,b}, T_{a,b}: a,b \in {\mathbb Z}_N, a > b\}$ be the orthonormal basis for the real vector space 
of the anti-Hermitian $N \times N$ matrices given by $S_{a,a}=i\Delta_{a,a}$, $S_{a,b}=i(\Delta_{a,b}+\Delta_{b,a})/\sqrt 2$  
and $T_{a,b} = (\Delta_{a,b}-\Delta_{b,a})/\sqrt 2$ for $a > b$,
then   $P_{G_0}(S_{a,b})=S_{a,b} G_0 - G_0 S_{a,b}$
and $P_{G_0}(T_{a,b})=T_{a,b} G_0 - G_0 T_{a,b}$ provides
a Parseval frame $\{ P_{G_0}(S_{a,b}), P_{G_0}(T_{a,b})\}_{a,b=1}^N$
for the tangent space $T_{G_0}\M$.
\end{thm}

\begin{proof}
We first note that because $S_{a,b}$ and $T_{a,b}$ are anti-Hermitian, $G_0 S_{a,b} G_0 + G_0 S_{a,b}^* G_0 = 0$
and $G_0 T_{a,b} G_0 + G_0 T_{a,b}^* G_0 = 0$, which shows the simplified expressions for the projections
onto the tangent space. Next, we show the Parseval property.
Since $\{S_{a,b}, T_{a,b}\}_{a,b=1}^N$ is an orthonormal basis,
the orthogonal projection $P_{G_0}$ maps it to a Parseval frame
for its span. This means we only need to show that
the span of the projected vectors is the space of all tangent vectors
at $G_0$.

Conjugating the orthonormal basis vectors 
$\{S_{a,a}: a \in {\mathbb Z}_N\} \cup \{S_{a,b}, T_{a,b}: a,b \in {\mathbb Z}_N, a > b\}$
with a unitary $U$ does not change the span. We choose $U$ so that it diagonalizes $G_0$, with the first $K$
columns of $U$ belonging to eigenvectors of $G$ of eigenvalue one. Thus
\begin{align*}(I-G_0) US_{a,b} U^* G_0 + G_0 U^* S_{a,b}^* U (I-G_0) = & U \left(\begin{array}{cc} 0 & 0\\ 0 & I_{N-K} \end{array}\right)
S_{a,b} \left(\begin{array}{cc} I_K & 0\\ 0 & 0 \end{array}\right) U^* \\
& - U \left(\begin{array}{cc} I_K & 0\\ 0 & 0 \end{array}\right)S_{a,b} \left(\begin{array}{cc} 0 & 0\\ 0 & I_{N-K} \end{array}\right) U^*
\end{align*}
where $I_K$ and $I_{N-K}$ are identity matrices
of size $K\times K$ and $(N-K)\times(N-K)$. Inserting the definition of $S_{a,b}$
shows that this is zero unless $a>K$ and $b \le K$. In that case,
$$
   (I-G_0) US_{a,b} U^* G_0 - G_0 U^* S_{a,b} U (I-G_0)
   = U (i \Delta_{a,b} - i \Delta_{b,a}) U^*/\sqrt 2 = i U T_{a,b} U^* \, .
$$
Similarly, if  $a>K$ and $b \le K$, then
$$
   (I-G_0) UT_{a,b} U^* G_0 - G_0 U^* T_{a,b} U (I-G_0)
   = -i U S_{a,b} U^* \, .
$$  
The set $\{i U T_{a,b} U^*, -i U S_{a,b} U^*\}_{a >K, b \le K}$   
is by inspection the orthonormal basis of a $2K(N-K)$-dimensional real
vector space of Hermitian matrices. Since this is 
in the range of $P_{G_0}$, it is a subspace of the tangent space.  Its dimension
then 
shows that the set $\{i U T_{a,b} U^*$, $-i U S_{a,b} U^*\}_{a >K, b \le K}$ 
spans the entire tangent space. Consequently, 
$
  \{P_{G_0}(S_{a,a}): a \in {\mathbb Z}_N\} \cup 
  \{P_{G_0}(S_{a,b}), P_{G_0}(T_{a,b}): a,b \in {\mathbb Z}_N, a > b\}
$
is a Parseval frame for the tangent space.
\end{proof}

An analogous theorem  holds for the real case.
\begin{thm} \label{thm:ProjT}
Suppose $\mathbb F = \mathbb R$ and let $ \{ T_{a,b}: a,b \in {\mathbb Z}_N, a > b\}$ be the orthonormal basis for the real vector space 
of the anti-symmetric $N \times N$ matrices given by $T_{a,b} = (\Delta_{a,b}-\Delta_{b,a})/\sqrt 2$ for $a > b$,
then $P_{G_0}(T_{a,b})=T_{a,b} G_0 - G_0 T_{a,b}$ provides
a Parseval frame $\{  P_{G_0}(T_{a,b})\}_{a,b=1}^N$
for the tangent space $T_{G_0}\M$.
\end{thm}
\begin{proof} The proof follows verbatim the proof of the complex case,
with $\{S_{a,b}\}_{a\ge b}$ omitted from the basis of the anti-Hermitian matrices.
We note that after conjugating with a suitable orthogonal matrix $U$,
the resulting projection of $T_{a,b}$, with $a>K$ and $b \le K$,
onto the tangent space is 
$$
  (I-G_0) UT_{a,b} U^* G_0 - G_0 U^* T_{a,b} U (I-G_0)
   = -i U S_{a,b} U^* \, 
$$
which is indeed a real symmetric matrix. Dimension counting then gives
that the image of $\{T_{a,b}\}_{a>b}$ is a basis for the $K(N-K)$-dimensional
space of tangent vectors at $G_0$.
\end{proof}

The
appearance of anti-Hermitian (respectively anti-symmetric) matrices is natural if one considers 
that selecting $G_0 \in \M$ and a differentiable function $u\in C^1({\mathbb R},U(N))$
with values in $U(N)$ (respectively $O(N)$), the manifold of $N\times N$ unitary (respectively orthogonal) matrices, induces
curves
in $\M$ of the form
$$
   \gamma(t) = u(t) G_0 u^*(t) \, .
$$
If $u(0)=I$ then  
from $\frac{d}{dt} u(t)u^*(t) = 0$, we see that
$\dot u(0)+\dot u^*(0)=0$, so $A=\dot u(0)$ is anti-Hermitian (respectively anti-symmetric)  
and   
$$
  \dot\gamma(0)= AG_0 + G_0  A^*=AG_0-G_0 A \, . 
$$
We denote the underlying map by $\Pi_{G_0}: U(N) \to \M$, $\Pi_{G_0}(U)=UG_0U^*$ (respectively  $\Pi_{G_0}: O(N) \to \M$, $\Pi_{G_0}(U)=UG_0U^*$)
and consider the associated tangent map $D\Pi_{G_0}$.

We recall that the embedding of the tangent spaces $T_I U(N)$ (respectively $T_I O(N)$) and of $T_{G_0}\M$ 
in ${\mathbb F}^{N\times N}$ induces via the Hilbert-Schmidt
inner product a Riemannian structure on the tangent spaces.

\begin{cor}
The tangent map $D\Pi_{G_0}$ from $T_IU(N)=\{A \in {\mathbb F}^{N\times N}, A = -A^*\}$ (respectively $T_I O(N)=\{A \in {\mathbb F}^{N\times N}, A = -A^\intercal\}$) 
to $T_{G_0}\M$ given by
$$ 
   D\Pi_{G_0}(A) = A G_0 -G_0 A
$$
is a surjective partial isometry.   
\end{cor}
\begin{proof}
The preceding theorems show
that the map $D\Pi_{G_0}$ is the synthesis  operator of a Parseval frame,
so it is
a surjective partial isometry.
\end{proof}

In order to characterize fixed points of the gradient flows associated
with each potential, we lift frame potentials and gradients to
the manifold of unitary (respectively orthogonal) matrices.

Given a function $\Phi: \M \to \mathbb R $ and $G_0\in \M$, we consider
the lifted function
$$
   \widehat \Phi_{G_0}: U(N) \to \mathbb R, \quad \widehat \Phi(U) = \Phi \circ\Pi_{G_0}(U)
   =\Phi(UG_0U^*) \, 
$$
when $\mathbb F = \mathbb C$
or
$$
   \widehat \Phi_{G_0}: O(N) \to \mathbb R, \quad \widehat \Phi(U) = \Phi \circ\Pi_{G_0}(U)
   =\Phi(UG_0U^\intercal) \, 
$$
when $\mathbb F = \mathbb R$.

\begin{cor}\label{cor0gradchar}
Let $\Phi: \M \to \mathbb R$ be differentiable, then
the gradient $\nabla \Phi(G_0)=0$ if and only if
$\nabla \widehat \Phi_{G_0}(I)=0$.
\end{cor}
\begin{proof}
We first cover the complex case. 
Letting $\gamma(t) = u(t)G_0 u^*(t)$,
$u(0)=I$ and $\dot u(0)=A=-A^*$, then the chain rule gives
$\frac{d}{dt}\vert_{t=0} \Phi(\gamma(t)) = \nabla \Phi(G_0)\cdot (AG_0-G_0A)$.
On the other hand,
$
\frac{d}{dt}\vert_{t=0} \Phi(\gamma(t)) 
= 
\frac{d}{dt}\vert_{t=0} \widehat \Phi_{G_0} (u(t))=
\nabla \widehat\Phi(I) \cdot A.
$
We conclude
$$
   \nabla \widehat \Phi_{G_0}(I) \cdot A = \nabla \Phi(G_0) \cdot D\Pi_{G_0}(A) \, 
$$
for any anti-Hermitian $A$. 
Thus, if $\nabla \Phi(G_0) =0$ then  $\nabla \widehat\Phi_{G_0}(I) =0$. Conversely,
since $D\Pi_{G_0}$ is surjective, $\nabla \widehat \Phi_{G_0}(I) = 0$
implies that $\nabla \Phi(G_0) = 0$.  

In the real case, $A^*$ is simply the transpose of $A$, so $A=-A^*$ means
that $A$ is skew-symmetric rather than anti-Hermitian.  
The same argument as in the complex case applies.
\end{proof}

\subsection{Frame potentials and properties of their critical points}

From the last part of this  section we have learnt that the gradient descent for any real-analytic frame potential 
always approaches a critical point of the frame potential. Next, we
direct our attention to the geometric character of the critical points corresponding to several choices
of frame potentials. An essential tool for the characterization of critical
points is that by the last corollary,  $\nabla \Phi$ vanishes at $G_0$ if
and only if $\nabla \hat \Phi_{G_0}$ vanishes at $I$.

We start with 
$\nabla (\widehat E^\alpha_{x,y})_G(I)$. From here on, when computing the gradient $\nabla \widehat \Phi_G(I)$
corresponding to any frame potential $\Phi$, we suppress the subscript $G$ and the argument $I$ and simply write $\nabla \widehat \Phi$.

\begin{lemma}\label{lemgradExy}
Let $\mathbb F= \mathbb C$ or $\mathbb F= \mathbb R$.  Let $G \in \M$, $\a \in (0, \infty)$ and  $x,y \in \{1,2,...,N \}$, 
and let
$ 
  (\widehat E^\alpha_{x,y})(U) = E^\alpha_{x,y}(UGU^*) \, ,
$
then the $(a,b)$ entry of the gradient of $\widehat{E}^\alpha_{x,y}$ at $I$ 
is given as follows:
\begin{align*}
[\nabla \widehat E^\alpha_{x,y} ]_{a,b} = 
\alpha   e^{\alpha |G_{x,y}|^2} (G_{y,x}G_{x,b}\delta_{y,a}  - G_{a,y}G_{y,x}\delta_{b,x} +G_{x,y}G_{y,b}\delta_{a,x}-G_{a,x} G_{x,y} \delta_{b,y}) \, .
\end{align*}
In particular, if $x=y$, then
\begin{align*}
[\nabla \widehat E^\alpha_{x,x} ]_{a,b} 
&=  2 \alpha  G_{x,x} e^{\alpha |G_{x,x}|^2} (\delta_{a,x} G_{x,b}-\delta_{b,x} G_{a,x}) \, .
\end{align*}
\end{lemma}

\begin{proof} 
Let $S_{a,b}=i(\Delta_{a,b}+\Delta_{b,a})/\sqrt 2$ for $a \leq b$ and let $T_{a,b}=(\Delta_{a,b}-\Delta_{b,a})/\sqrt 2$ for $a<b$ as before.

We first compute the entries of the matrices $S_{a,b} G - G S_{a,b}$ and $T_{a,b} G - G T_{a,b}$,
\begin{align*}
  [S_{a,b} G - G S_{a,b}]_{x,y} & = \frac{i}{\sqrt 2} [ \Delta_{a,b} G + \Delta_{b,a} G  - G \Delta_{a,b} - G \Delta_{b,a} ]_{x,y}\\
  & = \frac{i}{\sqrt{2}} [ \delta_{a,x} G_{b,y} + \delta_{b,x} G_{a,y} - G_{x,a} \delta_{b,y} - G_{x,b} \delta_{y,a}] \, .
\end{align*}
and
\begin{align*}
  [T_{a,b} G - G T_{a,b}]_{x,y} & = \frac{1}{\sqrt 2} [ \Delta_{a,b} G - \Delta_{b,a} G  - G \Delta_{a,b} + G \Delta_{b,a} ]_{x,y}\\
  & = \frac{1}{\sqrt{2}} [ \delta_{a,x} G_{b,y} - \delta_{b,x} G_{a,y} - G_{x,a} \delta_{b,y} + G_{x,b} \delta_{a,y}] \, .
\end{align*}

  Let $x, y \in {\mathbb Z}_N$, 
then
\begin{align*}
   \nabla \widehat E^\alpha_{x,y} \cdot S_{a,b} 
   &= \frac{d}{dt} \vert_{t=0} E^\alpha_{x,y} (G + t(S_{a,b}G-G S_{a,b}))\\
   &=\frac{i}{ \sqrt{2} } \alpha  e^{a|G_{x,y}|^2} (G_{y,x} (\delta_{a,x} G_{b,y} + \delta_{b,x} G_{a,y}
                              - G_{x,a} \delta_{b,y} - G_{x,b} \delta_{y,a})\\
                              & \phantom{\mbox{space}} -G_{x,y}(\delta_{a,x} G_{y,b} + \delta_{b,x} G_{y,a} - G_{a,x} \delta_{b,y} - G_{b,x} \delta_{y,a}))
\end{align*}
and 
\begin{align*}
   \nabla \widehat E^\alpha_{x,y} \cdot T_{a,b} 
   &= \frac{d}{dt} \vert_{t=0} E^\alpha_{x,y} (G_0 + t(T_{a,b}G_0-G_0 T_{a,b}))\\
   &= \frac{1}{\sqrt{2}} \alpha  e^{a|G_{x,y}|^2} (G_{y,x} (\delta_{a,x} G_{b,y}-\delta_{b,x} G_{a,y} - G_{x,a} \delta_{b,y} + G_{x,b} \delta_{a,y})\\
                                       & \phantom{\mbox{space}}         +G_{x,y} (\delta_{a,x} G_{y,b} -\delta_{b,x} G_{y,a} - G_{a,x} \delta_{b,y} + G_{b,x} \delta_{a,y})) \, .
\end{align*}
Thus, when $\mathbb F= \mathbb C$, summing the components of the gradient gives
\begin{align*}
[\nabla \widehat E^\alpha_{x,y} ]_{a,b} & =
 [(\nabla \widehat E^\alpha_{x,y} \cdot S_{a,b}) S_{a,b} + (\nabla \widehat E^\alpha_{x,y} \cdot T_{a,b}) T_{a,b} ]_{a,b}\\
&=  \alpha   e^{\alpha |G_{x,y}|^2} (G_{y,x}G_{x,b}\delta_{y,a}  - G_{a,y}G_{y,x}\delta_{b,x} +G_{x,y}G_{y,b}\delta_{a,x}-G_{a,x} G_{x,y} \delta_{b,y}) \, .
\end{align*}
Using the fact that $G_{j,l} = G_{l,j}$ for all $j,l \in \Zn$ when $\mathbb F= \mathbb R$, we obtain again
\begin{align*}
[\nabla \widehat E^\alpha_{x,y} ]_{a,b} & =
 [ (\nabla \widehat E^\alpha_{x,y} \cdot T_{a,b}) T_{a,b} ]_{a,b}\\
&=  \alpha   e^{\alpha |G_{x,y}|^2} (G_{y,x}G_{x,b}\delta_{y,a}  - G_{a,y}G_{y,x}\delta_{b,x} +G_{x,y}G_{y,b}\delta_{a,x}-G_{a,x} G_{x,y} \delta_{b,y}) \, .
\end{align*}

\end{proof}

Because the expression for $\nabla \widehat E^\alpha_{x,y}$ does not depend on whether $\mathbb F= \mathbb C$ or $\mathbb F= \mathbb R$, we do not distinguish between the two cases for the remaining gradient computations.

\subsubsection{The sum potential and the absence of orthogonal frame vectors}\label{sectionusum}  

Next, we investigate the sum potential.

\begin{prop}\label{propgradUsum}
Let $G \in \M$, let $a, b \in \Zn$, and let
$ 
  \wUsum(U) = \Usum(UGU^*) \, ,
$ then the $(a,b)$ entry of the gradient of $\wUsum$  is given as follows:
\begin{align*}
\left[ \nabla \wUsum  \right]_{a,b}   & = 2 \sum_{j \in \Zn,
j \not\in \{a,b\}} ( e^{\eta |G_{a,j}|^2}- e^{\eta |G_{b,j}|^2 } ) G_{a,j} G_{j,b} + 2  e^{\eta|G_{a,b}|^2} (G_{a,b} G_{b,b} -
G_{a,a} G_{a,b})
\\
& \phantom{=}\, \, \, \jh{+2 
e^{\eta(C_{N,K}^2 - \frac{K^2}{N^2})} \left(
e^{\eta G_{a,a}^2} G_{a,a} G_{a,b}    -    e^{\eta G_{b,b}^2}    G_{a,b} G_{b,b} \right)\, .}
\end{align*}
\end{prop}

\begin{proof}
By linearity of the gradient operator, we have
$$
\nabla \wUsum  =
\frac{1}{\eta} \left[ \tiny{\sum_{
\begin{array}{cc}
l,j=1\\
l \neq j 
\end{array}
}^N } \nabla \widehat{E}_{l,j}^\eta 
+
e^{\eta \left(  C_{N,K}^2 - \frac{K^2}{N^2} \right)}
\sum_{
s=1
}^N \nabla \widehat{E}_{s,s}^\eta 
\right]  .
$$
By applying 
Lemma~\ref{lemgradExy}, the claim follows.
\end{proof}

In the investigation of gradient descent for equal-norm frames, nontrivially
orthodecomposable frames presented undesirable critical points \cite{Strawn:mfdstruc}. We show that this class
of frames does not pose problems for our optimization strategy
when an initial condition is met. 

\begin{defn}
A frame $\mathcal F$ for a Hilbert space $\mathcal H$ is called \textem{orthodecomposable}
if there are mutually disjoint subsets $J_1$, $J_2$, \dots, $J_m$ 
partitioning ${\mathbb Z}_N$
and subspaces ${\mathcal H}_1$, ${\mathcal H}_2, \dots$, ${\mathcal H}_m$
of $\mathcal H$ such that
$\{f_j\}_{j \in J_k}$ is a frame for ${\mathcal H}_k$
and ${\mathcal H}_k \perp {\mathcal H}_l$ for all $k \ne l$,
so $\mathcal H = \bigoplus_{k=1}^m {\mathcal H}_k$.
\end{defn}

In terms of its Gram matrix $G$, a frame $\mathcal F$ is nontrivially orthodecomposable
if there is some permutation matrix $P$ which makes $G$ block diagonal,
$$
   G'=PGP^* = \left(\begin{array}{cc} G'_{1,1} & 0\\ 0 & G'_{2,2} \end{array}\right)
$$
where $G'_{1,1} \not \equiv 0$ and $G'_{2,2} \not\equiv 0$.

A sufficiently small initial value of the sum potential rules out that the gradient
descent on $\M$ encounters such orthodecomoposable frames.

\begin{prop}\label{propnozeros}
Let $G \in \M$ and $\eta>0$.  Suppose that $$\Usum (G) < 
%
 2 + (N^2-2) e^{\eta( K/(N^2-2) - K^2/N(N^2-2) + K(N-K)/(N(N-1)(N^2-2))} \, ,
$$
then $G$ contains no zero entries.
\end{prop}
\begin{proof} We prove the contrapositive.
Let $G_{j,l}= 0$.
Without loss of generality, we can assume that $j \ne l$ because if a diagonal entry in $G$ vanishes, then so do all entries
in the corresponding row.
Now we can perform Jensen's inequality for the entries other than $G_{j,l}$ and $G_{l,j}$ and obtain
$$
\Usum(G) \ge  
2+ (N^2-2)  e^{\eta/(N^2-2) \sum_{j,l=1}^N (|G_{j,l}|^2 - \eta \delta _{j,l} (K^2/N^2 - C_{N,K}^2))} \, .
$$
Inserting the value for $C_{N,K}$ and using the Parseval property gives the claimed bound.
\end{proof}

\subsubsection{The diagonal potential and equal-norm Parseval frames}\label{sectionudiag}

\begin{defn}\label{defUdiag} 
Let $G = (G_{a,b})_{a,b=1}^N  \in \M$.  Given $\delta \in (0,\infty)$, we define  the \textem{diagonal potential} $\Udiag: \M \rightarrow \R$ by
\begin{align*}
\Udiag(G) 
&= \frac{1}{\delta} \sum_{j=1}^N E_{j,j}^\delta (G)-  \frac{N}{\delta} e^{\delta \frac{K^2}{N^2}} \\
&= \frac{1}{\delta} \sum_{j=1}^N e^{\delta |G_{j,j}|^2} - \frac{N}{\delta} e^{\delta \frac{K^2}{N^2}} \, .
\end{align*}
\end{defn}

\begin{prop}\label{propgradUdiag}
Let $G \in \M$, let $a, b \in \Zn$, and let $ 
  \wUdiag(U) = \Udiag(UGU^*)
$.  Then the $(a,b)$ entry of the gradient of $\wUdiag$ is given as follows: 
\\
\\
\begin{align*}
\left[ \nabla \wUdiag   \right]_{a,b}   & = 2 G_{a,b} (  G_{a,a} e^{\delta |G_{a,a}|^2 } - G_{b,b} e^{ \delta |G_{b,b}|^2}) \, .
\end{align*}
\end{prop}

\begin{proof}
Observe that by linearity of the gradient operator, we have
$$
\nabla \wUdiag =
\nabla \left[
\frac{1}{\delta} \tiny{\sum\limits_{\begin{array}{cc} x \in \Zn \end{array}} }
\widehat E_{x,x}^{\delta}
\right]  =
\tiny{ \frac{1}{\delta}\sum\limits_{\begin{array}{cc} x \in \Zn \end{array}} }
\nabla \widehat E_{x,x}^\delta .
$$
By summing over the different cases for $x$ and applying Lemma~\ref{lemgradExy} , the claim follows.
\end{proof}

\begin{prop}\label{propequalnorm}
  Let $\Phi^\delta := \Psi + \Udiag$ be a function on $\M$, where $\Psi : \M \rightarrow [0, \infty)$ is any real analytic function that does not depend on the parameter $\delta$.  Let $G \in \M$ with no zero entries and suppose $\grad \Phi^\delta (G) = 0$ for all $\delta \in I$, where $I \subseteq (0, \infty)$ is an open interval, then $G$ is 
  the Gram matrix of an equal-norm Parseval frame.  

\end{prop}

\begin{proof}
Recall from Proposition~\ref{propgradUdiag}  that each $(a,b)$ entry of $\nabla \wUdiag$ is given by 
$$
[ \nabla \wUdiag ]_{a,b} 
=
 2 G_{a,b} (G_{a,a} e^{\delta |G_{a,a}|^2 } - G_{b,b} e^{ \delta |G_{b,b}|^2}).
$$
Thus, by hypothesis and Corollary~\ref{cor0gradchar}, we have
$$
[ \nabla \widehat \Phi^\delta ]_{a,b} 
=
 [\nabla \widehat \Psi]_{a,b} + 2 G_{a,b} (G_{a,a} e^{\delta |G_{a,a}|^2 } - G_{b,b} e^{ \delta |G_{b,b}|^2})
= 0 , 
\, \, 
\forall a,b \in \Zn, \forall \delta \in I
.
$$
Since $[ \nabla \widehat \Phi^\delta]_{a,b} $ is constant for all $\delta \in I$ and since $\Psi$ does not depend on $\delta$, taking the derivative of this expression with respect to $\delta$ yields
\begin{align*}
\frac{d}{d \delta}[ \nabla \widehat \Phi^\delta ]_{a,b} 
&=
\frac{d}{d \delta} [\nabla \widehat \Psi  ]_{a,b} + \frac{d}{d \delta} 2 G_{a,b} (G_{a,a} e^{\delta |G_{a,a}|^2 } - G_{b,b} e^{ \delta |G_{b,b}|^2})
\\
&= \, \, \, \, \, \, \, \, \, \, 0 
\, \, \, \, \, \, \, \,  \, \, \, \, \, \, \, \, \, \, \, \, \, \, 
+
 2 G_{a,b} ( G_{a,a}^3 e^{\delta |G_{a,a}|^2} - G_{b,b}^3 e^{\delta |G_{b,b}|^2} )
\\
&= 0
\end{align*}
for all $a,b \in \Zn$ and for all $\delta \in I$.
Since $G$ contains no zero entries, we can cancel the factor $2 G_{a,b}$ in these equations to obtain
$$
G_{b,b}^3 e^{\delta G_{b,b}^2} =  G_{a,a}^3 e^{\delta G_{a,a}^2},
$$
for all $ a,b \in \Zn$ and all $ \delta \in I$.
By the strict monotonicity of the function $x \mapsto x^3 e^{\delta x^2}$ on $\mathbb R_{+}$,
this implies $G_{a,a} = G_{b,b}$ for all $a, b$ in ${\mathbb Z}_N$.
This is only possible if $G$ is equal-norm, so we are done.
\end{proof}

\subsubsection{The chain potential and equipartitioning}\label{sectionuch}

\begin{defn}\label{defRx} 
\label{defLx} 
\label{defUch} 
Let $G = (G_{a,b})_{a,b=1}^N  \in \M$.  
Given $x  \in \Zn$ and $\a, \b \in (0,\infty)$, we define  the \textem{exponential row sum potential} $R_{x}^{\a, \b}: \M \rightarrow \R$ by
\begin{align*}
R_{x}^{\a,\b} (G)  
&= \frac{1}{\a} \sum\limits_{j=1, \j \neq x}^N E_{x,j}^\a (G) + \b E_{x,x}^1 (G)\\
&= \frac{1}{\a} \sum\limits_{j=1, \j \neq x}^N  e^{\a |G_{x,j}|^2} + \b  e^{ |G_{x,x}|^2}
\, .
\end{align*}
We define  the \textem{link potential} $L_{x}^{\a, \b}: \M \rightarrow \R$ by
\begin{align*}
L_{x}^{\a,\b} (G)  
&=(  R_{x}^{\a, \b}(G) - R_{x+1}^{\a, \b}(G) )^2
\end{align*}
and the \textem{chain potential} $\Uch: \M \rightarrow \R$ by
\begin{align*}
\Uch(G) 
&= \sum_{j \in \Zn} L_j^{\a,\b} (G) \, .
\end{align*}
\end{defn}

Next, we compute the gradient of $\widehat R^{\alpha,\beta}_x$ at $I$.

\begin{lemma}\label{lemgradRx}
Let $G \in \M$, $\a \in (0, \infty)$ and  $x \in \{1,2,...,N \}$, and
let
$
   \widehat R^{\alpha,\beta}_x(U) = R^{\alpha,\beta}_x(UGU^*) ,
$  then the $(a,b)$ entry of the gradient of $\widehat{R}^{\alpha}_x$ at $I$ 
is given as follows: 
\\
\\
\begin{align*}
\left[ \nabla \widehat{R}^{\alpha,\beta}_x  \right]_{a,b} &= 
\tiny{\sum\limits_{\begin{array}{cc}j \in \Zn \\ j \neq a  \end{array}}} 
e^{\alpha |G_{x,j}|^2}
\left[
- G_{a,j} G_{j,x} \delta_{b,x}
+ G_{x,j} G_{j,b} \delta_{a,x}
- G_{a,x} G_{x,j} \delta_{b,j}
\right] \\
&+ 
2 \beta e^{ |G_{x,x}|^2}
\left[
G_{x,x} G_{x,b} \delta_{x,a}
- G_{a,x} G_{x,x} \delta_{b,x}
\right] \, .
\end{align*}
\end{lemma}

\begin{proof}
This computation follows immediately from Lemma~\ref{lemgradExy} by observing that, because of linearity of the gradient operator, we have
\begin{align*}
\nabla \widehat{R}^{\alpha,\beta}_x  
= \nabla \left[  \frac{1}{\a} \tiny{\sum\limits_{\begin{array}{cc}j \in \Zn \\ j \neq a  \end{array}}} \widehat E_{x,j}^{\a}  + \b \widehat E_{x,x}^{1} \right] 
= \frac{1}{\a} \tiny{\sum\limits_{\begin{array}{cc}j \in \Zn \\ j \neq a  \end{array}}} \nabla \widehat E_{x,j}^{\a}   + \b  \nabla \widehat E_{x,x}^{1} .
\\
\end{align*} 

\end{proof}

\begin{lemma}\label{lemgradLx}
Let $G \in \M$ and $\a, \b \in (0, \infty)$.  Furthermore, let  $x, a  \in\Zn$ and set $b=a+1$. Let
$
 \widehat L^{\alpha,\beta}_x(U) = L^{\alpha,\beta}_x(UGU^*) \, ,
$
then the $(a, b)$-entry (ie, along the superdiagonal) of the gradient of $\widehat L^{\alpha,\beta}_x$ at $I$ is given as:
\begin{align*}
\left[ \nabla \widehat L^{\alpha,\beta}_x  \right]_{a,b} 
&= 2( \widehat R_x^{\a,\b}  - \widehat R_{x+1}^{\a,\b} ) 
\\
& \times \left\{
\tiny{\sum\limits_{\begin{array}{cc}j \in \Zn \\ j \neq a  \end{array}}} 
\left[e^{\alpha |G_{x,j}|^2}
\left(
- G_{a,j} G_{j,x} \delta_{b,x}
+ G_{x,j} G_{j,b} \delta_{a,x}
- G_{a,x} G_{x,j} \delta_{b,j}
\right)
\right. 
\right. \\
&    \indent \indent - \left.
e^{\alpha |G_{x+1,j}|^2}
\left(
- G_{a,j} G_{j,x+1} \delta_{b,x+1}
+ G_{x+1,j} G_{j,b} \delta_{a,x+1}
- G_{a,x+1} G_{x+1,j} \delta_{b,j}
\right)
\right] 
\\
&+ 
2 \beta 
\left[
e^{ |G_{x,x}|^2}
\left(
G_{x,x} G_{x,b} \delta_{x,a}
- G_{a,x} G_{x,x} \delta_{b,x}
\right)
\right.
\\
&\left . \left.
\indent 
 -
e^{ |G_{x+1,x+1}|^2}
\left(
G_{x+1,x+1} G_{x+1,b} \delta_{x+1,a}
- G_{a,x+1} G_{x+1,x+1} \delta_{b,x+1}
\right)
\right]
\vphantom{\tiny{\sum\limits_{\begin{array}{cc}j \in \Zn \\ j \neq a  \end{array}}} } \right\} \, .
\end{align*}
\end{lemma}

\begin{proof}

If we let $h(t) = t^2$, then we see that $$\Lx (G) =( \Rx (G) - R_{x+1}^{\a,\b} (G) )^2 = h( \Rx (G) - R_{x+1}^{\a,\b} (G) ) .$$  Therefore, by applying the chain rule and linearity of the gradient operator, we see that
\begin{align*}
\nabla \widehat L^{\alpha,\beta}_x  
&= h'( R^{\alpha,\beta}_x (G)  - R_{x+1}^{\a,\b}  (G)) \nabla ( \widehat
R^{\alpha,\beta}_x - \widehat R_{x+1}^{\a,\b} ) 
\\
&= 2(  R^{\alpha, \beta}_x (G)  -  R_{x+1}^{\a,\b} (G) ) ( \nabla 
\widehat R^{\alpha,\beta}_x - \nabla \widehat R_{x+1}^{\a,\b} ) 
\end{align*}
The rest follows by Lemma~\ref{lemgradRx}.
\end{proof}

For notational convenience, we define $\varphi_\alpha : \Zn^3 \times \M \rightarrow \C$  for $\alpha>0$ by
\begin{align*}
\varphi_\alpha (a,b,x,G) 
&=
 \tiny{\sum\limits_{\begin{array}{cc}j \in \Zn \\ j \neq a  \end{array}}} 
\left[e^{\alpha |G_{x,j}|^2}
\left(
- G_{a,j} G_{j,x} \delta_{b,x}
+ G_{x,j} G_{j,b} \delta_{a,x}
- G_{a,x} G_{x,j} \delta_{b,j}
\right)
\right. \\
&    \indent \indent - \left.
e^{\alpha |G_{x+1,j}|^2}
\left(
- G_{a,j} G_{j,x+1} \delta_{b,x+1}
+ G_{x+1,j} G_{j,b} \delta_{a,x+1}
- G_{a,x+1} G_{x+1,j} \delta_{b,j}
\right)
\right] \, .
 \end{align*}

\begin{prop}\label{propgradUch}
Let $G \in \M$, $\a, \b,  \in (0, \infty)$.  Let $a \in \Zn$ and set $b=a+1$, so that $(a,b)$ entry of $G$ falls on the superdiagonal, and let $ 
  \wUch^{\alpha,\beta}(U) = \Uch(UGU^*),
$ 
then the $(a,b)$ entry of the gradient of $\wUch^{\a,\b}$ is given as follows: 
\\
\\
\begin{align*}
\left[ \nabla \wUch^{\alpha,\beta}   \right]_{a,b} 
&= 
2  \sum\limits_
{
\tiny{\begin{array}{cc}
j \in \Zn 
\end{array}
}
}
\left( R_j^{\a,\b} (G) -   R_{j+1}^{\a,\b} (G) \right) \varphi_\alpha(a,b,x, G)  \\
&+
4 \beta \left\{ -  G_{a,a} G_{a,b} e^{|G_{a,a}|^2} [  R_{a-1}^{\a,\b} (G) -   R_{a}^{\a,\b} (G)] \right. \\
&\indent \indent + (G_{a,a} G_{a,b} e^{|G_{a,a}|^2}  +  G_{a,b} G_{b,b} e^{|G_{b,b}|^2}   ) [  R_{a}^{\a,\b} (G) -   R_{b}^{\a,\b} (G)]  \\
&\left. \indent \indent - G_{a,b} G_{b,b} e^{|G_{b,b}|^2}  [  R_{b}^{\a,\b} (G) -   R_{b+1}^{\a,\b} (G)]  \right\} \, .
\end{align*}

\end{prop}

\begin{proof}
Observe that
$$
\nabla \wUch^{\alpha,\beta}  =
\nabla \left[
\tiny{\sum\limits_{\begin{array}{cc} j \in \Zn \end{array}} }
\widehat L_j^{\a,\b}
\right] =
\tiny{\sum\limits_{\begin{array}{cc} j \in \Zn \end{array}} }
\nabla \widehat L_j^{\a,\b}.
$$
By summing over the different cases for $j$ and using Lemma~\ref{lemgradLx}, the claim follows, where we have isolated the nonzero terms which are multiplied by the parameter $\beta$ (i.e., corresponding to $j=a-1$, $j=a$, and $j=a+1$).
\end{proof}


\begin{prop}\label{propalphaequipart}
Let $\Phi^{\a, \b} := \Psi + \Uch$ be a function on $\M$, where $\Psi : \M \rightarrow [0, \infty)$ is any real analytic function that does not depend on the  parameters $\a$ or $\b$.  Let $G \in \M$ be equal-norm with no zero entries, fix $\a\in (0, \infty)$, and suppose that $\grad \Phi_{ch}^{\a, \b} (G) = 0$ for all $\beta \in J $, where $J \subseteq (0, \infty)$ is an open interval, then $G$ is $\a$-equipartitioned.

\end{prop}

\begin{proof}
Since $\Psi$ does not depend on $\b$ and since $\nabla \Phi_{ch}^{\a, \b} =0$ for all $\beta \in J$, then using corollary~\ref{cor0gradchar} and taking the partial derivative with respect to $\b$ of an $(a,b)$ entry of $\nabla \widehat \Phi_{ch}^{\a, \b}$ gives
\begin{align*}
\frac{d}{d \b} [\nabla \widehat \Phi_{ch}^{\a, \b}  ]_{a,b}
&=
\frac{d}{d \b}  [\nabla \widehat \Psi (G) ]_{a,b}       +      \frac{d}{d \b}   [\nabla \wUch  ]_{a,b}
\\
&=
\, \, \, \, \, \, \, \, \, \, 0 \, \, \, \, \, \, \, \, \, \, \, \,   \, \, \, \, \, \, \, \, \, \, \, \,        +      \frac{d}{d \b}   [\nabla \wUch^{\a,\b} ]_{a,b}
\\
&= 0
\end{align*}
for all $\b \in J$.
In particular, we have $\frac{d}{d \b}   [\nabla \wUch^{\a,\b} ]_{a,b} = 0 $ for all $a,b \in \Zn$ and for all $\b \in J$.  

 Next, we  compute $\frac{d}{d \b}   [\wUch^{\a,\b} ]_{a,b}$ for the case where $b=a+1$ (ie, along the superdiagonal), thereby inducing a set of equations which will lead to the desired result. So, from here on, we suppose that $b=a+1$. 

First, we observe a simplification that results from the assumption that $G$ is equal-norm.  Referring back to Proposition~\ref{propgradUch}, we note that every additive term of  $[\nabla \wUch ]_{a,b}$ has a factor of the form $(R_j^{\a, \b}(G) - R_{j+1}^{\a,\b} (G) ).$  However, since $G$ is equal-norm, we can replace each of these factors with $(R_j^{\a}(G) - R_{j+1}^{\a} (G) )$ to denote the fact that the $\beta$ terms corresponding to the diagonal entries have canceled because of the equal-norm property.  After doing this, we see that there are only three terms of $[\nabla \wUch^{\a,\b}  ]_{a,b}$ which still depend on $\beta$.  Now the desired partial derivative is easy to compute, which yields
\begin{align*}
\frac{d}{d \b}   [\nabla \wUch  ]_{a,b}
&=
- 4 G_{a,a} G_{a,b} e^{|G_{a,a}|^2} [  R_{a-1}^{\a,\b} (G) -   R_{a}^{\a,\b} (G)] \\
& \indent +4 (G_{a,a} G_{a,b} e^{|G_{a,a}|^2}  +  G_{a,b} G_{b,b} e^{|G_{b,b}|^2}   ) [  R_{a}^{\a,\b} (G) -   R_{b}^{\a,\b} (G)]  \\
& \indent - 4G_{a,b} G_{b,b} e^{|G_{b,b}|^2}  [  R_{b}^{\a,\b} (G) -   R_{b+1}^{\a,\b} (G)  \\
&=0
\end{align*}
Once again, because $G$ is equal-norm, it follows that $G_{a,a} = G_{b,b} = \frac{K}{N}$, so this equation can be rewritten as 
\begin{align*}
\frac{d}{d \b}   [\nabla \wUch  ]_{a,b}
&=
-4 \frac{K}{N} G_{a,b} e^{ |\frac{K}{N}|^2} \left(  R_{a-1}^\a (G) - 3R_{a}^\a (G) + 3R_{b}^\a(G) - R_{b+1}^\a (G)   \right)
\\
&=0 \, .
\end{align*}
Since $G$ contains no zero entries, we can cancel the factor(s) $-4 \frac{K}{N} G_{a,b} e^{ |\frac{K}{N}|^2}$ to obtain
$$R_{a-1}^\a (G) - 3R_{a}^\a (G) + 3R_{b}^\a(G) - R_{b+1}^\a (G) =0$$
for all $a,b \in \Zn$, where $b=a+1$.  This induces the linear system $Ax = 0$, where $A$ is the  $N \times N$ circulant matrix
$$
A = \left(\begin{array}{ccccccc}
-3 & 3 & - 1 & 0     & \cdots & 0 & 1\\
 1  &-3& 3    & - 1  & \cdots & 0 & 0\\
0 & 1 & -3 & 3     & \cdots & 0 & 0\\
 \vdots &  \vdots &  \vdots     &  \vdots   &  \vdots  &  \vdots  &  \vdots \\
3 & -1 & 0 & 0     & \cdots & 1 & -3\\
\end{array} \right)
$$
and 
$$
x
=
\left(
\begin{array}{cc}
R^\a_1 (G) \\
R^\a_2 (G)\\
\vdots \\
R^\a_N(G)  
\end{array}
\right) .
$$
The circulant matrix $A$ is the polynomial 
$A=-3I+3S-S^2+S^{N-1}$ of the cyclic shift matrix $S$.
Therefore, its eigenvectors coincide with those of $S$,
$$
v_j = \left(
\begin{array}{cc}
1 \\
\omega_j \\
\vdots \\
\omega_j^{N-1}
\end{array}
\right)
$$  
and by the spectral theorem the eigenvalues of $A$ are then given by
$$
\lambda_j = -3 + 3 \omega_j -  \omega_j^2 +\omega_j^{N-1},\, \, j \in \Zn,
$$
where
$\omega_j= e^{\frac{2 \pi i j }{N}}$, the $N$th roots of unity, are the corresponding eigenvalues of $S$.

The system $Ax=0$ is homogeneous, so we would like to obtain the zero eigenspace of $A$.  By letting $j \in \Zn$, setting $\lambda_j=0$, and then factoring, we obtain
\begin{align*}
0 
&= -3 + 3 \omega_j -  \omega_j^2 +\omega_j^{N-1} \\
&=  -(1 - \omega_j ) ( 2 - \omega_j - \omega_j^{N-1}
 ) \, .
\end{align*}
Inspecting both factors, we see that $\lambda_j = 0$ iff $\omega_j=1$ iff $j \equiv 0 \mod N$.  Thus, the zero eigenspace is $1$-dimensional and spanned by the vector of all ones.  In particular, this shows that
$$
R^\a_1 (G) =
R^\a_2 (G) =
\cdots \\
= R^\a_N(G).  
$$
In other words, $G$ is $\a$-equipartitioned.
\end{proof}

\begin{cor}\label{corequidistribution}
Let $\Phi^{\a, \b} := \Psi + \Uch$ be a function on $\M$, where $\Psi : \M \rightarrow [0, \infty)$ is any real analytic function that does not depend on the parameters $\a$ or $\b$.  Let $G \in \M$ be equal-norm with no zero entries, and, furthermore, suppose that $\grad \Phi_{ch}^{\a, \b} (G) = 0$ for all $\a \in I$ and all $\beta \in J$, where $I, J \subseteq (0, \infty)$ are open intervals, then $G$ is equidistributed.

\end{cor}

\begin{proof}
Since $J$ is an open interval, it follows by proposition \ref{propalphaequipart} that $G$ is $\a$-equipartitioned for every $\a \in I$.  Therefore, by 
Proposition~\ref{propequipartfonopen}, $G$ is equidistributed.
\end{proof}

\subsubsection{A characterization of equidistributed frames}\label{sectionumain}

Finally, we combine these definitions to obtain the family of potential functions which will yield our main theorem, as defined below.

\begin{defn}\label{defUmain} 
Let $G = (G_{a,b})_{a,b=1}^N  \in \M$.  Given $\a, \b, \delta, \eta \in (0,\infty)$, we define  the \textem{combined potential}, $\Umain: \M \rightarrow \R$, by
\begin{align*}
\Umain(G) 
&= \Uch(G) + \Udiag(G) + \Usum(G). 
\end{align*}
\end{defn}

\jh{
\begin{defn}\label{deffamwise}
Let  $G \in \M$, let $I, J, T \subseteq (0,\infty)$ 
be open intervals and
$\eta >0$, then we say that $G$ is a \textem{family-wise critical point} 
with respect to $\{\Phi^{\alpha,\beta,\delta,\eta}\}_{\alpha \in I, \beta \in J, \delta \in T}$
if $\nabla \Umain (G) = 0$ for all $\a \in I$, $\b \in J$, and $\delta \in T$.
\end{defn}
}

\jh{

\begin{thm}\label{thmmain} 
Let $G \in \M$, let $I, J, T \subseteq (0,\infty)$ 
be open intervals and $\eta>0$.  If $G$ is a family-wise critical point with respect to $\{\Phi^{\alpha,\beta,\delta,\eta}\}_{\alpha \in I, \beta \in J, \delta \in T}$
and $$\Phi^{\a', \b', \delta', \eta}(G) <2 + (N^2-2) e^{\eta( K/(N^2-2) - K^2/N(N^2-2) + K(N-K)/(N(N-1)(N^2-2))}$$ for some $\a', \b', \delta' \in (0, \infty)$, then $G$ is equidistributed.
\end{thm}

\begin{proof}
Since $G$ is a family-wise critical point, $\nabla \Umain (G) = 0$ for all $\a \in I, \b \in J$ and $\delta \in T$.

Since $\Phi^{\a', \b', \delta', \eta}(G) = \Phi_{ch}^{\a', \b'}(G) + \Phi_{diag}^{\delta'} (G) + \Phi_{sum}^{\eta}(G) < D$, 
with 
$$ D=2 + (N^2-2) e^{\eta( K/(N^2-2) - K^2/N(N^2-2) + K(N-K)/(N(N-1)(N^2-2))}$$
then
it must also be the case that $\Usum(G) < D$.  Hence, $G$ contains no zero entries, by Proposition~\ref{propnozeros}.  

Now, with respect to Proposition~\ref{propequalnorm}, we can momentarily rewrite $\Umain$ as $\Phi^\delta = \Psi + \Udiag$, where $\Psi =\Uch+ \Usum$.  Since we have confirmed that there are no zero entries, since $\Psi$ does not depend on $\delta$ and since $\grad \Phi^\delta (G) = 0$ for all $\delta \in T$, it follows from Proposition~\ref{propequalnorm} that $G$ corresponds to an equal-norm Parseval frame. 

Finally, with respect to Corollary~\ref{corequidistribution}, we can once again rewrite $\Umain$ as $\Phi^{\a,\b} = \Psi + \Uch$, where $\Psi = \Usum+ \Udiag $ this time.  Since we have confirmed that $G$ contains no zeros, since $G$ is equal-norm, $\Psi$ does not depend on $\a$ or $\b$, and  $\grad \Phi^{\a, \b}_{ch} (G) = 0$ for all $\a \in I$ and  $\b \in J$, it follows from Corollary \ref{corequidistribution} that $G$ is equidistributed.

 \end{proof}
 }
 
 Alternatively, we can relax the 
 requirement 
on the value of the potential 
 and simply assume that our family-wise critical point contain no zero entries.  It is clear from the preceding proof that this would also yield equidistributivity, as stated below.

\jh{
\begin{cor}\label{cormain} 
Let $\eta>0$.  If $G \in \M$ is a family-wise critical point 
with respect to the family of frame potentials
$\{\Phi^{\alpha,\beta,\delta,\eta}\}_{\alpha \in I, \beta \in J, \delta \in T}$
and $G$ contains no zero entries, then $G$ is equidistributed.
\end{cor}
}

\jh{
\begin{prop}\label{propUsumenough}
Let $\eta>0$ and $G \in \M$ be equidistributed.  If $\nabla \Usum (G) = 0$, then $G$ is a family-wise critical point
with respect to $\{\Phi^{\alpha,\beta,\delta,\eta}\}_{\alpha, \beta, \delta \in (0,\infty)}$.
\end{prop}
}

\begin{proof}
Since $G$ is equidistributed, we see immediately that $\Uch(G) = 0$ for all $\a, \b \in (0, \infty)$, so it follows that $\nabla \Uch(G) = 0$ for all $\a, \b \in (0, \infty)$.  Similarly, since $G$ must also be equal-norm, we have $\Udiag(G) = 0$ for all $\delta \in (0, \infty)$, which implies$\nabla \Udiag(G) = 0$ for all $\delta \in (0, \infty)$.  Thus, if  $\nabla \Usum (G) = 0$, then 
$$\nabla \Umain (G) = \nabla \Uch(G) + \Udiag(G) + \nabla \Usum (G) = 0 $$
for all $\a, \b, \delta \in (0, \infty)$, so the claim follows.
\end{proof}

The assumption in the preceding proposition is met when the frame is generated with a group representation as specified below.


\begin{prop}\label{lemcircFP}
Suppose $\Gamma$ is a finite group of size $N=|\Gamma|$ with a unitary representation $\pi: \Gamma \to B(\mathcal H)$ 
on the complex $K$-dimensional Hilbert space $\mathcal H$ and $\{f_g = \pi(g) f_e\}$
is an $(N,K)$-frame. If the Gram matrix $G$ satisfies $G_{g,h} = G_{h^{-1},g^{-1}}$ for all $g,h \in \Gamma$, then $\nabla \Usum (G) = 0$ for all $\eta \in (0, \infty)$.
\end{prop}

\begin{proof}
Fix $\eta \in (0, \infty)$.  Since $G$ is equidistributed (see Example \ref{exharmonic}), it is equal norm, so the last two additive two terms from the $(a,b)$ gradient entry in Proposition~\ref{propgradUsum} cancel. Therefore, to show that the gradient vanishes, it is sufficient to show that
$$
       \sum_{j \in \Gamma} e^{\eta |G_{a,j}|^2} G_{a,j} G_{j,b} =  \sum_{j \in \Gamma}  e^{\eta |G_{b,j}|^2 }  G_{a,j} G_{j,b}
$$
for all $a, b \in \Gamma$.
As a first step, we note that the group representation gives $G_{x,y} = \langle f_y, f_x\rangle = \langle \pi(x^{-1} y) f_e, f_e \rangle
\equiv H(x^{-1} y) $. Thus, we can change the summation index and get
$$
    \sum_{j \in \Gamma} e^{\eta |G_{a,j}|^2} G_{a,j} G_{j,b} =  \sum_{j \in \Gamma} e^{\eta |H(j^{-1} a)|^2} H(j^{-1}a) H(b^{-1} j) =
    \sum_{j \in \Gamma} e^{\eta |H(j^{-1})|^2} H(j^{-1}) H(b^{-1} a j) \, .
$$
We also note that $|H(g)| = |H(g^{-1})|$, so in combination with changing the summation index we obtain
$$
\sum_{j \in \Gamma} e^{\eta |G_{a,j}|^2} G_{a,j} G_{j,b} = \sum_{j \in \Gamma} e^{\eta |H(j)|^2} H(j^{-1}) H(b^{-1} a j)
 = \sum_{j \in \Gamma} e^{\eta |H(j^{-1}b)|^2} H(b^{-1}j) H(b^{-1} a j^{-1} b) \, .
$$
Finally, using the fact that the 
Gram matrix has the assumed structure gives $H(h^{-1}g)=H(gh^{-1})$ for $h=a^{-1} b$ and $g=j^{-1} b$, which yields
$$
 \sum_{j \in \Gamma} e^{\eta |G_{a,j}|^2} G_{a,j} G_{j,b} = \sum_{j \in \Gamma} e^{\eta |H(j^{-1}b)|^2} H(b^{-1}j) H( j^{-1} a) =
   \sum_{j \in \Gamma} e^{\eta |G_{b,j}|^2} G_{a,j} G_{j,b} \, .
$$
This completes the proof, since $\eta$ was arbitrary.

\end{proof}

The claimed property of the Gramian is true if $\Gamma$ is abelian.

\begin{cor}Suppose $\Gamma$ is a finite abelian group of size $N=|\Gamma|$ with a unitary representation $\pi: \Gamma \to B(\mathcal H)$ 
on a real or complex $K$-dimensional Hilbert space $\mathcal H$ and $\{f_g = \pi(g) f_e\}$
is an $(N,K)$-frame, then $\nabla \Usum (G) = 0$ for every $\eta \in (0, \infty)$.
\end{cor}

There is an abundance of \jh{equidistributed Parseval} frames obtained with representations of abelian groups, in particular the harmonic
frames that exist for any combination of the number of frame vectors $N$ and dimension $K\le N$.

\begin{cor}\label{famwiseexist}
For every pair of integers $1\leq K \le N$ \jh{and for every $\eta>0$}, there exists a family-wise critical point 
with respect to $\{\Phi^{\alpha,\beta,\delta,\eta}\}_{\alpha, \beta, \delta \in (0,\infty)}$
on $\M$.
\end{cor}

\jh{
The gradient of the sum energy also vanishes for any Gramian corresponding to a mutually unbiased basic sequence which has been rescaled to admit Parsevality.

\begin{prop}\label{propMUBsareFPs}
If $G \in \M$ is the Gramian of a mutually unbiased basic sequence, then $\nabla \Usum(G) =0$ for all $\eta \in (0, \infty)$.  
\end{prop}
\begin{proof}
Fix $\eta \in (0, \infty)$. We recall that  the 
Gram matrix has diagonal entries that are equal to $K/N$,
and the other entries either vanish in diagonal blocks
or have the same magnitude in off-diagonal blocks.
Assuming the frame vectors
are grouped in $M$ subsets of size $L$, then
$N=ML$ and 
there are $M(M-1)L^2$ off-diagonal entries of the same magnitude 
$C_{M,L,K}$. 
Based on the Hilbert-Schmidt norm of $G$, we then have
$$
    C_{M,L,K} = \sqrt{\frac{K(ML-K)}{M^2(M-1)L^3}} \, .
$$

In order to make the block structure apparent in the notation, 
we write the matrix $G$ as $G = (G_{x,y}^{(p,q)})_{p,q \in \mathbb Z_M, x,y \in \mathbb Z_L}$, where the doubly-indexed superscript indicates in which block the entry is and the subscript indicates the position within the block. The absolute value of any entry then satisfies
$$|G_{x,y}^{(p,q)}| = 
\left\{
\begin{array}{cc}
\frac{K}{ML}, & x=y, p=q \\
0, & x \neq y, p = q \\
C_{M,L,K}, & p \neq q\, .
\end{array}
\right. 
$$
To see the claim, we verify that every entry of $\nabla \wUsum$ vanishes.  Since this is automatically true for the diagonal entries, let $(p,x), (q,y) \in \mathbb Z_S \times \mathbb Z_L$ with $(p,x) \neq (q,y)$.  One has that either $p=q$ or $p \neq q$. If $p=q$, then re-expressing the identity in Proposition~\ref{propgradUsum} in terms of block notation and noting that the last two terms on the right-hand side cancel due to the equal-norm property 
yields
\begin{align*}
\left[ \nabla \wUsum  \right]_{x,y}^{(p,p)}   
& = 2
\tiny{ \sum\limits_{
\begin{array}{cc}
t=1 \\
t \neq x,y
\end{array}
}^{L} 
}
( e^{\eta | G_{x,t}^{(p,p)}|^2}- e^{\eta | G_{y,t}^{(p,p)}|^2 } ) G_{x,t}^{(p,p)} G_{t,y}^{(p,p)} \\
&\phantom{=} +
2 \, 
\tiny{ \sum_{\begin{array}{cc}
s=1 \\
s \neq p 
\end{array}}^M  \sum\limits_{
\begin{array}{cc}
t=1 \\
t \neq x,y
\end{array}
}^{L} 
}
( e^{\eta | G_{x,t}^{(p,s)}|^2}- e^{\eta | G_{y,t}^{(p,s)}|^2 } ) G_{x,t}^{(p,s)} G_{t,y}^{(s,p)} \, .
\end{align*}
The first  series on the right-hand side
is zero because $G_{x,t}^{(p,p)}=G_{y,t}^{(p,p)} = 0$
since $t \not \in \{x,y\}$. The second series
also vanishes because when $s \ne p$, then
$|G_{x,t}^{(p,s)}|=|G_{y,t}^{(p,s)}| = C_{M,L,K}$. 
Thus, the block diagonal entries of the gradient
vanish.

On the other hand, if $p \neq q$, then we get 
\begin{align*}
\left[ \nabla \wUsum  \right]_{x,y}^{(p,q)}   
& = 2
\tiny{ \sum\limits_{
\begin{array}{cc}
s,t=1 \\
(s,t) \neq (p,x), (q,y)
\end{array}
}^{M,L} 
}
( e^{\eta | G_{x,t}^{(p,s)}|^2}- e^{\eta | G_{y,t}^{(q,s)}|^2 } ) G_{x,t}^{(p,s)} G_{t,y}^{(s,q)}\\
&=
 2
\tiny{ \sum\limits_{
\begin{array}{cc}
t=1 \\
t \neq x
\end{array}
}^{L} 
}
( e^{\eta | G_{x,t}^{(p,p)}|^2}- e^{\eta | G_{y,t}^{(q,p)}|^2 } ) G_{x,t}^{(p,p)} G_{t,y}^{(p,q)} \\
& \phantom{=}+ 2
\tiny{ \sum\limits_{
\begin{array}{cc}
t=1 \\
t \neq y
\end{array}
}^{L} 
}
( e^{\eta | G_{x,t}^{(p,q)}|^2}- e^{\eta | G_{y,t}^{(q,q)}|^2 } ) G_{x,t}^{(p,q)} G_{t,y}^{(q,q)}\\
&\phantom{=} +
 2 
 \tiny{ \sum\limits_{
\begin{array}{cc}
s =1 \\
s \neq p, q
\end{array}
}^{M} 
}
\tiny{ \sum\limits_{
\begin{array}{cc}
t=1
\end{array}
}^{L} 
}
( e^{\eta | G_{x,t}^{(p,s)}|^2}- e^{\eta | G_{y,t}^{(q,s)}|^2 } ) G_{x,t}^{(p,s)} G_{t,y}^{(s,q)} \, .
\end{align*}
The first series vanishes because
$G_{x,t}^{(p,p)} = 0$, the second one because $G^{(q,q)}_{t,y}=0$
and the last one because
$
  |G^{(p,s)}_{x,t}| = |G^{(q,s)}_{y,t}| = C_{M,L,K} \, .
$
This confirms that $\nabla \wUsum = 0$ and, since $\eta$ was arbitrary, the claim is proven.
\end{proof}
}

As a consequence of this Proposition and of Proposition~\ref{propUsumenough}, we know that Examples~\ref{Ex:GrassmannianSemicirc} and
\ref{Ex:GrassmannianMUB} provide us with family-wise critical 
points.

If the Gramian does not contain vanishing entries, then
we can characterize equidistributed frames, taking advantage of 
the fact that the term $\Usum$ in $\Umain$ is no longer needed in this case.

\begin{thm}\label{thmcharnozeros}
Let $G \in \M$ and suppose that $G$ contains no zero entries.  The
Gramian $G$ is equidistributed if and only if $\nabla[ \Udiag + \Uch] (G) = 0$ for all $\alpha \in I$, $\beta \in J$, and $\delta \in T$, where $I, J, T \subseteq (0, \infty)$ are open intervals.
\end{thm}

\begin{proof}
If $G$ is equidistributed, then $\Udiag(G)=0$ and $\Uch(G)=0$ for all $\a, \b, \delta \in (0,\infty)$, so it follows that $\nabla \Udiag(G)=0$ and $\nabla \Uch(G)=0$ for all $\a, \b, \delta \in (0,\infty).$  For the converse, since $G$ contains no zero entries, it follows by 
Proposition~\ref{propequalnorm} that $G$ is equal norm.  With this established, it then follows from Corollary~\ref{corequidistribution} that $G$ is equidistributed.

\end{proof}

\subsection{Constructing equidistributed Grassmannian Parseval frames}\label{sectiongrass}
\jh{ We conclude the discussion of the relation between frame potentials and the structure of optimizers by showing how an equidistributed Grassmannian equal-norm Parseval frame can be obtained as the limit of minimizers to the sequence $\{\Phi^{\eta_n}_{sum}\}_{n=1}^\infty$, where $\eta_n \to \infty$.}

\begin{prop}\label{equalnormetaseqgivesgrass}
Let $\{\eta_m\}_{m=1}^\infty$ be a positive, increasing sequence such that $\lim\limits_{m\rightarrow \infty} \eta_m= + \infty$ and suppose $\left\{G(m)=\left(G(m)_{a,b} \right)_{a,b=1}^N \right\}_{m=1}^\infty \subseteq \M$ is a sequence of Gram matrices  such that $\Phi_{sum}^{\eta_m}$ achieves its absolute minimum at $G(m)$ for every $m \in \{1,2,3,...\}$ and each $G(m)$ corresponds to an equal-norm frame, then there exists a subsequence $\left\{G(m_s)  \right\}_{s=1}^\infty$ and $G \in \M$ such that  $\lim\limits_{s\rightarrow \infty} G(m_s) = G$, where $G$ is the Grassmannian of an equal-norm Parseval frame.\
\end{prop}

\begin{proof}
For each $m \in \{1,2,...\}$, since $G(m)$ is equal norm, the diagonal part of $\Phi_{sum}^{\eta_m} (G(m))$ simplifies so that we can write
$$
   \Phi_{\mathrm{sum}}^{\eta_m} (G(m)) = \Phi_{\mathrm{od}}^{\eta_m} (G) + 
   \sum_{j=1}^N e^{\eta_m C_{N,K}^2}  \, .
$$
Furthermore, by Parsevality, since each $G(m)$ is equal norm, there must always exist an off diagonal entry $G(m)_{a,b}$ such that $|G(m)_{a,b}|^2  \geq C_{N,K}^2$.  Hence,  $\mu(G(m))=\max\limits_{a,b \in \Zn} |G(m)_{a,b}|$ for every $m$,  which allows us to replace $\Phi_{od}^{\eta_m}$ with $\Phi_{sum}^{\eta_m}$ in the proof strategy from Proposition \ref{propodgivesgrass}, which shows that $G$ corresponds to a Grassmannian Parseval frame.  Finally, since each $G(m)$ is equal-norm, it follows that $G$ must also be equal-norm.  Therefore, $G$ is a Grassmannian equal-norm Parseval frame.
\end{proof}

If the sequence of minimizing Parseval frames has the stronger property of being equidistributed, which implies that it is equal norm,
 then the limit of the corresponding subsequence 
is equidistributed as well.

\begin{prop}\label{corequidistributedgrass}
Let $\{\eta_m\}_{m=1}^\infty$ be a positive, increasing sequence such that $\lim\limits_{m\rightarrow \infty} \eta_m= + \infty$, and suppose $\left\{G(m)=\left(G(m)_{a,b} \right)_{a,b=1}^N \right\}_{m=1}^\infty \subseteq \M$ is a sequence of equidistributed Gramians such that  $\Phi_{sum}^{\eta_m}$ achieves its absolute minimum at $G(m)$ for every $m \in \{1,2,3,...\}$, then there exists a subsequence $\left\{G(m_s)  \right\}_{s=1}^\infty$ and $G \in \M$ such that  $\lim\limits_{s\rightarrow \infty} G(m_s) = G$, where $G$ corresponds to an equidistributed Grassmannian Parseval frame.\end{prop}

\begin{proof}
Since each $G(m)$ is equidistributed, we can define for each $\alpha \in (0, \infty)$ the sequence $s^\alpha(m):=\sum\limits_{j \in \Zn} e^{\a | G(m)_{a,j} |^2}$, where our definition is independent of the choice of $a \in \Zn$.  Since the entries of $\{G(m)\}_{m=1}^\infty$ converge to those of $G$, we have by continuity
$$\sum\limits_{j \in \Zn} e^{\a | G_{a,j} |^2}= \lim\limits_{m \rightarrow \infty} s^\a(m) = \sum\limits_{j \in \Zn} e^{\a | G_{b,j} |^2}$$
for all $a, b \in \Zn$.  Since this is true for every $\a \in (0, \infty)$, $G$ is equidistributed by Proposition \ref{propequipartfonopen}.
Additionally, being the limit of a sequence of minimizers for 
$\{\Phi_{sum}^{\eta_m}\}_{m \in \mathbb N}$, $G$ is also a Grassmannian
Parseval frame.
\end{proof}

If we know that if each $G(m)$ is a \jh{family-wise critical point} without vanishing entries, then we can 
characterize this limit in terms of the gradient of frame potentials.

\begin{thm}\label{thmequidistributedgrass}
Let $I,J,T$ be open intervals in $(0,\infty)$ and
let $\{\eta_m\}_{m=1}^\infty$ be a positive, increasing sequence such that $\lim\limits_{m\rightarrow \infty} \eta_m= + \infty$.
If  $\left\{G(m)=\left(G(m)_{a,b} \right)_{a,b=1}^N \right\}_{m=1}^\infty \subseteq \M$ is a sequence such that $\Phi_{sum}^{\eta_m}$  achieves its absolute minimum at $G(m)$,
each $G(m)$ contains no vanishing entries and $G$ is a
family-wise fixed point 
with respect to $\{\Phi^{\alpha,\beta,\delta,\eta_{m}}\}_{\alpha \in I, \beta \in J, \delta \in T}$, then there exists a subsequence $\left\{G(m_s)  \right\}_{s=1}^\infty$ and $G \in \M$ such that  $\lim\limits_{s\rightarrow \infty} G(m_s) = G$, 
 where $G$ is the Gram matrix of an equidistributed Grassmannian Parseval frame.\end{thm}

\begin{proof}
By Corollary \ref{cormain}, it follows that each $G(m)$ is equidistributed and hence equal norm.  Therefore, by Lemma \ref{equalnormetaseqgivesgrass} there exists a subsequence $\left\{G(m_s)  \right\}_{s=1}^\infty$ and $G \in \M$ such that  $\lim\limits_{s\rightarrow \infty} G(m_s) = G$, where $G$ corresponds to a Grassmannian equal-norm Parseval frame.  Finally, since every $G(m_s)$ is equidistributed, it follows by Proposition~\ref{corequidistributedgrass} that $G$ is also  equidistributed.
\end{proof}

The existence of equiangular Parseval frames for certain pairs of $N$ and $K$ provides an abundance of examples for which this theorem holds;  however, due to our current inability to verify when a non-equiangular critical point of $\Umain$ is at an absolute minimum, we are unable to state outright that non-equiangular, equidistributed frames exist which satisfy the conditions of Theorem \ref{thmequidistributedgrass}.  Based on numerical experiments, it is our conjecture that 
Example~\ref{ex46hankel} is an absolute minimizer of $\Umain$ for all $\eta \in (0,\infty)$ and therefore corresponds to a Grassmannian equal-norm Parseval frame
which is equidistributed.  

In addition, we know that the conclusion of the preceding theorem
can hold even if $G$ contains vanishing entries, as provided
  in examples of family-wise critical points
given by the equidistributed Grassmannian equal-norm Parseval frames in Examples~\ref{Ex:GrassmannianSemicirc} and
\ref{Ex:GrassmannianMUB}.

\section*{Acknowledgments}

Both authors
would like to express thanks for the hospitality of the American Institute of Mathematics
where the outline of this work took shape, and would like to thank Dustin Mixon for drawing
attention to the intriguing interplay between Grassmannian and equidistributed frames.
B. G. B. acknowledges the great hospitality of Gitta Kutyniok and her group at
the Technische Universit{\"a}t Berlin where this paper was completed.

\begin{appendix}

\section{ $\M$ as a real analytic manifold}\label{app:Mrealanalytic}

\begin{thm}\label{thmMisanalytic}
The manifold $\M$ is a real analytic submanifold of the (linear) manifold
of matrices ${\mathbb F}^{N \times N}$. The dimension of $\M$ is $K(N-K)$
if $\mathbb F=\mathbb R$ and $2K(N-K)$ if $\mathbb F = \mathbb C$.
\end{thm}
\begin{proof}
As before, we define $\M$ as the set of $N \times N$ orthogonal projections with rank $K$.  Given any $G_0 \in \M$, we can find a subset of indices, $J = \{ n_1, n_2, ..., n_K \} \subset \Zn$ of size $|J|=K$ such that
the rows of $G$ indexed by $J$ are linearly independent. By the orthogonality of $G$, removing
the rows and columns  corresponding to the indices in $\Zn \setminus J$ from $G_0$ then yields the Gramian $(G_0)^{J,J}$ of the row vectors indexed by $J$,
which is invertible since the rows are linearly independent.  By continuity of the determinant in the entries of a matrix,  
there exists $ \epsilon >0$ such that  for any $G \in \M \cap B(G_0; \epsilon)$  the $K \times K$ submatrix $G^{J,J}$  consisting of the rows from $G$ indexed by $J$
is invertible.  Now, for $G \in B(G_0; \epsilon) \cap \M$, consider the map
$$
    \tilde{\phi}_J: B(G_0; \epsilon) \cap \M \rightarrow {\mathbb F}^{N\times K} (\C) , G \mapsto  (G^{J,J})^{-1} G^{J, N} \, .
$$
Noting that $\tilde{\phi}_J (G)$ contains a $K \times K$ identity submatrix, we define the chart $\phi_J(G)$ to be 
the  $K \times (N-K)$ matrix given by $\tilde{\phi}_J (G) = (I_K \, \phi_J(G))$, thereby defining what will be our local coordinates in $\mathbb F^{K\times (N-K)}$.  Then $\tilde{\phi}_J$ is analytic, since the inverse of $G^{J,J}$ is rational in its entries; hence, $\phi_J$ is also analytic, since there is no loss of analyticity in the removal of entries.  

 To see that $\phi_J$ has an analytic inverse, we show that we can reconstruct $G$ from $\phi(G)$ in an analytic fashion.  First, we reinsert the $K \times K$ identity block in a way that corresponds to $J$ so that we have recovered the $K \times N$ matrix $A:= \tilde{\phi}_J (G) = (G^{J,J})^{-1} G^{J, N}$, as above.  
 Next, we form the $K \times K$ Gram matrix $Q = A A^* = (G^{J,J})^{-1} G^{J, N}  (G^{J, N})^* ((G^{J,J})^{-1})^*$.  Since $G_{J,N}$ was extracted from an orthogonal projection,  $G^{J,N} (G^{J,N})^* = G^{J,J}$, so that $Q= (G^{J,J})^{-1}$ is analytic in the coordinates. Next, we orthogonalize
 the rows of $A$ to obtain  $B : = Q^{-1/2} A=(G^{J,J})^{1/2}A$. The negative square root of $Q$ is seen to be analytic in $Q$ via
 a convergent power series expansion of $(cI-(cI-Q))^{-1/2}$ in terms of the powers of $cI-Q$, where $c>\|Q\|$. The rows of $B$ then provide an orthonormal basis
 with the same span as the rows of $A$ and $B B^*=I$.
 Thus, $B$ is the synthesis operator of a Parseval frame with the Gram matrix 
$$B^* B = ((Q^{-\frac{1}{2}} A) )^* Q^{-\frac{1}{2}} A = G^{N,J} (G^{J,J})^{-1} G^{J,N} = G.$$
 We see that the entries of $G$ are analytic 
in the coordinates if there is $c>0$ such that the power series expansion of $(cI-(cI-Q))^{-1/2}$ converges, so  $\phi_J^{-1}$ is  analytic on the range $\phi_J(B(G_0; \epsilon))$.

Combining the analyticity of the charts and of their inverses, we conclude that 
 $\M$ is a real analytic manifold because $\phi_J\circ \phi_L^{-1}$ is analytic on the image of the intersection of the domains of $\phi_J$ and $\phi_L$
 for any subsets $J$ and $L$ of size $|J|=|L|=K$. The dimension of $\M$ is
 the real dimension of $\mathbb F^{K\times (N-K)}$, which is
 $K(N-K)$
if $\mathbb F=\mathbb R$ and $2K(N-K)$ if $\mathbb F = \mathbb C$.
 \end{proof}

\end{appendix}

\providecommand{\href}[2]{#2}


\begin{thebibliography}{10}

\bibitem{BF03}
J.~J. Benedetto and M.~Fickus, \emph{Finite normalized tight frames}, Adv.
  Comput. Math. \textbf{18} (2003), no.~2-4, 357--385, Frames.

\bibitem{BK06}
J.~J. Benedetto and J.~D. Kolesar, \emph{Geometric properties of {G}rassmannian
  frames for ${R}^2$ and ${R}^3$}, EURASIP Journal on Advances in Signal
  Processing (2006), 1--17, Article ID 49850.

\bibitem{BodmannCasazza:2010}
B.~G. Bodmann and P.~G. Casazza, \emph{The road to equal-norm {P}arseval
  frames}, J. Funct. Anal. \textbf{258} (2010), 397--420.

\bibitem{BodPau05}
B.~G. Bodmann and V.~I. Paulsen, \emph{Frames, graphs and erasures}, Linear
  Algebra Appl. \textbf{404} (2005), 118--146.

\bibitem{CFM12}
P.~G. Casazza, M.~Fickus, and D.~G. Mixon, \emph{Auto-tuning unit norm frames},
  Appl. Comput. Harmon. Anal. \textbf{32} (2012), no.~1, 1--15.

\bibitem{CasKov03}
P.~G. Casazza and J.~Kova{\v{c}}evi{\'c}, \emph{Equal-norm tight frames with
  erasures}, Adv. Comput. Math. \textbf{18} (2003), no.~2-4, 387--430.

\bibitem{FINITEFRAMESBOOK}
P.~G. Casazza and G.~Kutyniok, \emph{Finite frames: Theory and applications},
  Birkh{\"a}user, 2013.

\bibitem{Chr03}
O.~Christensen, \emph{An introduction to frames and {R}iesz bases}, Applied and
  Numerical Harmonic Analysis, Birkh\"auser Boston, Inc., Boston, MA, 2003.

\bibitem{ORTHOPLEX}
J.~H. Conway, R.~H. Hardin, and N.~J.~A. Sloane, \emph{Packing lines, planes,
  etc.: Packings in {G}rassmannian spaces}, Experiment. Math. \textbf{5}
  (1996), no.~2, 139--159.

\bibitem{DuffinSchaffer}
R.~J. Duffin and A.~C. Schaeffer, \emph{A class of nonharmonic {F}ourier
  series}, Trans. Amer. Math. Soc. \textbf{72} (1952), no.~2, 341--366.

\bibitem{Strawn:mfdstruc}
K.~Dykema and N.~Strawn, \emph{Manifold structure of spaces of spherical tight
  frames}, J. Pure Appl. Math. \textbf{28} (2006), 217--256.

\bibitem{Elw11}
H.~J. Elwood, \emph{Constructing complex equiangular {P}arseval frames}, Ph.D.
  thesis, University of Houston, December 2011.

\bibitem{GR09}
C.~Godsil and A.~Roy, \emph{Equiangular lines, mutually unbiased bases, and
  spin models}, European J. Combin. \textbf{30} (2009), no.~1, 246--262.

\bibitem{GVT98}
V.~K. Goyal, M.~Vetterli, and N.~T. Thao, \emph{Quantized overcomplete
  expansions in {${\mathbb R}^N$}: Analysis, synthesis and algorithms}, IEEE
  Trans. Inform. Theory \textbf{44} (1998), 16--31.

\bibitem{HANLARSON}
D.~Han and D.~R. Larson, \emph{Frames, bases and group representations},
  Memoirs of the Amer. Math. Soc. \textbf{147} (2000), no.~697, x+94.

\bibitem{Heil11}
Ch. Heil, \emph{A basis theory primer}, expanded ed., Applied and Numerical
  Harmonic Analysis, Birkh\"auser/Springer, New York, 2011.

\bibitem{HP04}
R.~Holmes and V.~I. Paulsen, \emph{Optimal frames for erasures}, Lin. Alg.
  Appl. \textbf{377} (2004), 31--51.

\bibitem{MUBs-Ivan}
I.~D. Ivanovic, \emph{Geometrical description of quantum state determination},
  Journal Physics A \textbf{14} (1981), no.~12, 3241--3245.

\bibitem{Kal06}
D.~Kalra, \emph{Complex equiangular cyclic frames and erasures}, Linear Algebra
  Appl. \textbf{419} (2006), no.~2-3, 373--399.

\bibitem{KOVACEVICCHEBIRA}
J.~Kova{\v c}evi{\'c} and A.~Chebira, \emph{An introduction to frames}, Now
  Publishers Inc., 2008.

\bibitem{kurd:loj}
K.~Kurdyka and A.~Parusinski, \emph{Wf-stratification of sub-analytic functions
  and the {L}ojasiewicz inequality}, C. R. Acad. Sci. Paris S{\'e}r. Math.
  \textbf{318} (1994), 129--133.

\bibitem{loj:ensanal}
S.~{\L}ojasiewicz, \emph{Ensembles semi-analytiques}, I.H.E.S. Notes (1965),
  1--153.

\bibitem{Heath04}
D.~J. Love, R.~W. Heath, and T.~Strohmer, \emph{Grassmannian beamforming for
  multiple-input multiple-output wireless systems}, IEEE Trans. on Information
  Theory \textbf{49} (2003), 2734--2747.

\bibitem{merl:gradglows}
B.~Merlet and T.~N. Nguyen, \emph{Convergence to equilibrium for
  discretizations of gradient-like flows on {R}iemannian manifolds},
  Differential Integral Equations \textbf{26} (2013), no.~5/6, 487--674.

\bibitem{Okt07}
O.~Oktay, \emph{Frame quantization theory and equiangular tight frames}, Ph.D.
  thesis, Univ. Maryland at College Park, 2007.

\bibitem{RBKSC04}
J.~M. Renes, R.~Blume-Kohout, A.~J. Scott, and C.~M. Caves, \emph{Symmetric
  informationally complete quantum measurements}, J. Math. Phys. \textbf{45}
  (2004), no.~6, 2171--2180.

\bibitem{HEATH}
S.~Schwarz, Jr. R.~W.~Heath, and M.~Rupp, \emph{Adaptive quantization on a
  {G}rassmann-manifold for limited feedback beamforming systems}, IEEE Trans.
  on Signal Processing, \textbf{61} (2013), no.~18, 4450--4462.

\bibitem{MUBs-Schwing}
J.~Schwinger, \emph{Unitary operator bases}, Proc. Natl. Acad. Sci. USA (1960),
  570--579.

\bibitem{Strawn2012}
N.~Strawn, \emph{Optimization over finite frame varieties and structured
  dictionary design}, Appl. Comput. Harmon. Anal. \textbf{32} (2012), no.~3,
  413--434. \MR{2892742}

\bibitem{HS03}
T.~Strohmer and R.~Heath, \emph{Grassmannian frames with applications to coding
  and communcations}, Appl. Comput. Harmon. Anal. \textbf{14} (2003), 257--275.

\bibitem{STDH07}
M.~A. Sustik, J.~A. Tropp, I.~S. Dhillon, and R.W. Heath, Jr., \emph{On the
  existence of equiangular tight frames}, Linear Algebra Appl. \textbf{426}
  (2007), no.~2-3, 619--635.

\bibitem{welch:bound}
L.~R. Welch, \emph{Lower bounds on the maximum cross correlation of signals},
  IEEE Trans. on Information Theory \textbf{20} (1974), no.~3, 397--9.

\bibitem{Zau}
G.~Zauner, \emph{Quantendesigns -- {G}rundz{\"u}ge einer nichtkommutativen
  {D}esigntheorie}, Ph.D. thesis, Universit{\"a}t Wien, 1999.

\end{thebibliography}
\end{document}